\documentclass[11pt,reqno]{amsart}
\usepackage[utf8x]{inputenc}
\usepackage[english]{babel}
\usepackage{amsmath,amsfonts,amssymb,amsthm,array,tikz-cd,dsfont}
\usepackage{graphicx}
\usepackage[margin=2.5cm]{geometry}

\numberwithin{equation}{section}

\newtheorem{theorem}{Theorem}[section]
\newtheorem{proposition}[theorem]{Proposition}
\newtheorem{corollary}[theorem]{Corollary}
\newtheorem{lemma}[theorem]{Lemma}

\theoremstyle{definition}
\newtheorem{definition}[theorem]{Definition}

\newtheorem{remark}[theorem]{Remark}

\newtheorem{conjecture}[theorem]{Conjecture}

\long\def\forget#1\forgotten{}

\makeatletter
\@namedef{subjclassname@2020}{\textup{2020} Mathematics Subject Classification}
\makeatother

\newcommand{\1}{\mathds{1}}
\newcommand{\Z}{{\mathbb Z}}
\newcommand{\R}{\mathbb{R}}
\newcommand{\N}{\mathbb{N}}

\newcommand{\f}{\varphi}
\newcommand{\e}{\varepsilon}
\newcommand{\M}{\mathcal{M}}
\newcommand{\F}{\mathcal{F}}
\renewcommand{\L}{\mathcal{L}}

\newcommand{\supp}{\mathrm{supp}}
\newcommand{\conv}{\mathrm{Conv}}
\newcommand{\dom}{\mathrm{dom}}
\begin{document}
\title{On Mahler's conjecture for even s-concave functions 
in dimensions 1 and 2.}
\author{Matthieu Fradelizi, Elie Nakhle}


%\subjclass[2010]{52A20, 52A40,  53A15, 52B10.}
%\keywords{Mahler conjecture, Blaschke-Santal\'o inequality, functional form, Legendre transform, equipartitions, log-concave functions}

%\nnfootnote{Keywords: Mahler conjecture, Blaschke-Santal\'o inequality, functional form, Legendre transform, equipartitions, s-concave functions\\Mathematics Subject Classification 2020 - 52A20, 52A40,  53A15, 52B10}

\begin{abstract}
In this paper, we establish different sharp forms of Mahler's conjecture for $s$-concave  even functions in dimensions $n$, for $n=1$ and $2$, for $s>-1/n$, thus generalizing our previous results in \cite{FN} on log-concave even functions in dimension 2, which corresponds to the case $s=0$.
The functional volume product of an even $s$-concave function $g$ is 
\[
\int_{\R^{n}}g(x)dx\int_{\R^{n}}\L_{s}g(y)dy,
\]
where $\L_{s}g$ is the $s$-polar function associated to $g$. The analogue of Mahler's conjecture for even $s$-concave functions postulates that this quantity is minimized for the indicatrix of a cube for any $s>-1/n$.
In dimension $n=1$, we prove this conjecture for all $s\in(-1,0)$ (the case $s\ge0$ was established by the first author and Mathieu Meyer in \cite[page 17]{FM10}). In dimension $n=2$, we consider the case $1/s\in\Z$: for $s>0$, we establish Mahler's conjecture for general $s$-concave even functions; for $s<0$, we prove a sharp inequality for $s$-concave functions $g$ such that $g^s$ admits an asymptote in every direction. Notice that this set of functions is quite natural to consider, when $s<0$, since it is the largest subset of $s$-concave functions stable by $s$-duality. 
\end{abstract}
\maketitle

\section{Introduction.}
The classical Blaschke-Santal\'o inequality gives the following sharp relation between the volume of a centrally symmetric convex body $K$ in~$\R^{n}$ and the volume of its polar body $K^{\circ}$: 
$$P(K)=|K||K^{\circ}|\leq P(B_{2}^{n}),$$
where $B_{2}^{n}$ is the Euclidean ball of radius one, $K^{\circ}=\{y\in\R^{n}, \langle x,y\rangle\leq 1, \forall x\in K\}$ is the polar body of $K$ in $\R^{n}$ and $|A|$ stands for the Lebesgue measure of a Borel subset~$A$ of $\R^{n}$.\\

One of the well-known problems in the field of geometry of convex bodies is an inverse form of the Blaschke-Santal\'o inequality, namely Mahler's conjecture. This conjecture states that the product of the volume of a centrally symmetric convex body $K$ and the volume of its polar body~$K^{\circ}$, should be minimized by the cube, that is, 
$$
P(K)\geq P([-1,1]^{n})=\frac{4^{n}}{n!}.
$$
This conjecture was stated in 1938 by Mahler \cite{Mah2}. 
%It was later conjectured that the equality case occurs if and only if $K$ is a Hanner polytope, which is a centrally symmetric polytope obtained by repeatedly taking the Cartesian product or the $l_{1}$-sum of segments. 
The inequality in dimension two was established by Mahler himself \cite{Mah1}, the $3$-dimensional inequality and the equality case were proved by Iriyeh and Shibata \cite{IS}, a shorter proof can be found in \cite{FHMRZ}. The conjecture was also proved for several particular families of convex bodies like unconditional convex bodies by Saint Raymond \cite{SR,M}, zonoids by Reisner \cite{Re,GMR}, hyperplane sections of $B_{p}^{n}=\{x\in\R^{n};\sum |x|_{i}^{p}\leq 1\}$ by Karasev \cite{K}.\\

Functional forms of Mahler's conjecture have also been studied, where the centrally symmetric convex body is replaced with an even $s$-concave function and the polar body is replaced with the  suitable $s$-polar of this function. 
%Let $g:\R^{n}\to\R_{+}$ be an even log-concave function. This means that for some even convex function $\varphi:\R^{n}\to\R\cup\{+\infty\}$, one has 
More precisely, let us recall the definition of $s$-concavity.

\begin{definition}
Let $s\in\R$. A function $g:\R^{n}\to\R_{+}$ is $s$-concave, if for every $\lambda\in[0,1]$, and every $x$, $y\in\R^{n}$ such that $g(x)g(y)>0$, one has
$$
g((1-\lambda)x+\lambda y)\geq( (1-\lambda)g(x)^{s}+\lambda g(y)^{s})^{\frac{1}{s}}, \quad \hbox{for}\ s\neq0, 
$$
and $g((1-\lambda)x+\lambda y)\geq g(x)^{1-\lambda}g(y)^\lambda$, for $s=0$.
\end{definition}
Let $g:\R^{n}\to\R_{+}$ be an $s$-concave function not identically zero. Notice that for $s>0$, it means that $g^{s}$ is concave on its support, while for $s<0$, it means that $g^{s}:\R^n\to\R_+\cup\{+\infty\}$ is convex.
 For $s$-concave functions, the $s$-duality takes the following form. We define, for every $y\in\R^{n}$, 
$$
\L_{s}g(y):=\inf_{x\in\R^{n}}\frac{\left(1-s\langle x,y\rangle\right)_{+}^{\frac{1}{s}}}{g(x)},\quad \hbox{for}\ s\neq0,\  \hbox{and}\quad  \L_{0}g(y)=\inf_{x\in\R^{n}}\frac{e^{-\langle x,y\rangle}}{g(x)},
$$
where the infimum are considered on the set $\{x\in\R^{n}; g(x)>0\}$. Note that $\L_{s}g$ is $s$-concave and that $\L_{0}$ can also be described using the Legendre transform: for any function $\f:\R^n\to\R\cup\{+\infty\}$, one has $\L_{0}(e^{-\f})=e^{-\f^*}$, where $\f^*(y)=\sup_x(\langle x,y\rangle-\f(x))$ denotes the Legendre transform of $\f$.

%We define the $s$-functional volume product of an $s$-concave function $g$ to be
%\[\int_{\R^{n}}g(x)dx\int_{\R^{n}}\L_{s}g(y)dy.\]
The following conjecture is the natural analogue of Mahler's conjecture for even $s$-concave functions. 
\begin{conjecture}\label{mahler-s}
Let $s>-\frac{1}{n}$ and $g:\R^{n}\to\R_{+}$ be an even $s$-concave function such that $0<\int g<+\infty$. Then,
$$
\int_{\R^{n}}g(x)dx\int_{\R^{n}}\L_{s}g(y)dy\geq\frac{4^{n}}{(1+s)\cdots(1+ns)},
$$
with equality for $g=\1_{[-1,1]^n}$.
\end{conjecture}
The conjecture has been proved for $s=0$ and $g$ being unconditional in \cite{FM08a, FM08b}. Recall that a function is unconditional if $g(|x_{1}|,...,|x_{n}|)=g(x_{1},...,x_{n})$ for every $(x_{1},...,x_{n})\in\R^{n}$. It has also been established for $s>0$ and $n=1$ in \cite[page 17]{FM10}; for $s>0$ and $g$ unconditional in \cite{BF,FGSZ} and for $s=0$ and $n=2$ in \cite{FN}. 

According to the sign of $s$, there are two ways of reformulating Conjecture \ref{mahler-s}. If $s>0$, then it is equivalent to finding the exact lower bound for the quantity 
\[
\int_{\R^{n}}f(x)^{m}dx\int_{\R^{n}}(\L_1 f(y))^{m}dy,
\]
where $f$ is an even concave function such that $0<\int f^m<+\infty$. While for $s<0$, Conjecture~\ref{mahler-s} is equivalent to finding the exact lower bound for 
\[
\int_{\R^{n}}\frac{1}{f(x)^{m}}dx\int_{\R^{n}}\frac{1}{(\M f(y))^{m}}dy,
\]
where $f:\R^n\to (0,+\infty]$ is even, convex, $f$ goes to infinity at infinity and for every $y\in\R^{n}$
\[
\M f(y)=\sup_{x\in\R^{n}}\frac{1+\langle x,y\rangle}{f(x)}.
\]

In our first main result, we prove Conjecture~\ref{mahler-s} in the case where  $1/s\in\N$, and in dimension $n=2$.
\begin{theorem}\label{thm:spos}
    Let $s>0$ be such that $1/s\in\N$ and let $g:\R^{2}\to\R_{+}$ be an even $s$-concave function such that $0<\int g<+
    \infty$. Then,
    \[
    \int_{\R^{2}}g(x)dx\int_{\R^{2}}\L_{s}g(y)dy\geq\frac{16}{(1+s)(1+2s)},
    \]
with equality for $g=\1_{[-1,1]^2}$.
\end{theorem}
Notice that letting $s\to0$, we get a new proof of the log-concave case from \cite{FN}, though without the equality case, see Corollary \ref{cor:FN23}. \\

For $s<0$, the situation is more involved, because the $s$-dual $\L_s g$ of an integrable $s$-concave function $g$ is not only $s$-concave but, as observed in \cite{Ro, FGSZ}, the function $(\L_{s}g)^s$ has also asymptotes in every direction. For example, for $g=\1_{[-1,1]^n}$, one has $(\L_{s}g(y))^s=1-s\|y\|_1$, where $\|y\|_1=\sum_{i=1}^n|y_i|$, for $y=(y_1,\dots,y_n)$. More precisely, in general, one has $(\L_{s}g)^s\in\F$, where the set $\F$ is defined below. 

\begin{definition}\label{def:F}
$\F$ is the set of convex lower semi-continuous functions $f:\R^{n}\to (0,+\infty)$ such that, for any $x\neq 0$, one has $\lim_{t\to+\infty}f(tx)=+\infty$,  and $t\mapsto\frac{f(tx)}{t}$ is non-increasing on $(0,+\infty)$.
\end{definition}
Notice that if $g^s\in\F$, then it is not difficult to see that $0<\int g<+\infty$, see the inequality  \eqref{eq:gint} in section 2.
Moreover, it is easy to see that for any $g$ such that $g^s\in\F$, one has $\L_s\L_s g=g$, from which it follows that for any map $g$, one has $\L_s\L_s\L_s g=\L_sg$. 
As we shall see from the one-dimensional case, it is natural to formulate two more conjectures from which Conjecture \ref{mahler-s} follows in the case $s<0$. The first one postulates a sharp lower bound on the functional volume product of $s$-concave even functions, among the ones such that $g^s\in\F$.

\begin{conjecture}\label{mahler-s-F}
Let $s\in(-\frac{1}{n},0)$ and let $g:\R^{n}\to\R_{+}$ be even such that $g^s\in\F$. Then,
$$
\int_{\R^{n}}g(x)dx\int_{\R^{n}}\L_{s}g(y)dy\geq\frac{1}{(1+ns)}\times\frac{4^{n}}{(1+s)\cdots(1+ns)},
$$
with equality for $g(x)=(\max(1,\|x\|_\infty))^{1/s}$.
\end{conjecture}
Notice that $(\max(1,\|x\|_\infty))^{1/s}=\L_s\L_s(\1_{[-1,1]^n})$.
In \cite{FGSZ}, the above conjecture was established if the function $g$ is unconditional. In particular, the conjecture holds for $n=1$. In our second main result, we establish Conjecture \ref{mahler-s-F} when $n=2$ and $-1/s\in\N$.

\begin{theorem}\label{th:mahler-s-F}
Let $s\in(-\frac{1}{2},0)$ be such that $-1/s\in\N$ and let $g:\R^{2}\to\R_{+}$ be even  such that $g^s\in\F$. Then,
$$
\int_{\R^{2}}g(x)dx\int_{\R^{2}}\L_{s}g(y)dy\geq\frac{16}{(1+s)(1+2s)^2},
$$
with equality for $g(x)=(\max(1,\|x\|_\infty))^{1/s}$.
\end{theorem}

The second conjecture compares the integral of $g$ with the integral of its double $s$-polar.
\begin{conjecture}\label{comp-bipolar}
Let $s\in(-\frac{1}{n},0)$ and let $g:\R^{n}\to\R_{+}$ be even $s$-concave such that $0<\int g<+\infty$. Then,
\begin{equation}\label{eq:bipolar}
\int_{\R^{n}}g(x)dx\ge(1+ns)\int_{\R^{n}}\L_s\L_{s}g(x)dx,
\end{equation}
with equality for $g=\1_{[-1,1]^n}$.
\end{conjecture}
Notice that, for any function $g:\R^{n}\to\R_{+}$, one has, for any $x\in\R^n$, $g(x)\le \L_s\L_sg(x)$ thus $\int_{\R^{n}}g(x)dx\le\int_{\R^{n}}\L_s\L_{s}g(x)dx$. The conjectured inequality \eqref{eq:bipolar} postulates that one has an opposite bound up to constant for $s$-concave even functions.

 %This statement has also been proved in more general situations \cite{AKM,FM07,L1,L2}.\\

%The symmetric version of Mahler's conjecture for even convex functions which was proved in dimension $1$ by \cite{FM08a} and in dimension $2$ by \cite{FN} states the following.
%\begin{conjecture}\label{conj:conv}
%Let $n\ge1$ and $\f:\R^n\to\R\cup\{+\infty\}$ be an even convex function such that $0<\int_{\R^{n}}e^{-\f(x)}dx<+\infty$. Then,
%\[
%\int e^{-\f}\int e^{-\f^{*}}\ge 4^n,
%\]
%with equality if and only if there exists  $c\in\R$ and two Hanner polytopes  $K_1\subset F_1$ and $K_2\subset F_2$, where $F_1$ and $F_2$ are two supplementary subspaces in $\R^n$, such that for all $(x_1,x_{2})\in F_1\times F_{2}$
%\[\f(x_1+x_2)=c+\|x_1\|_{K_1}+I_{K_2}(x_2),\quad a.e.\]
%where $I_{K}$ is the function defined by $I_{K}(x)=0$ if $x\in K$ and $I_{K}(x)=+\infty$ if $x\notin K$.
%\end{conjecture}
%with equality if and only if there exists a partition $\{1,...,n\}=I_{1}\cup I_{2}$ and two Hanner polytopes $K_{1}\subset F_{1}$ and $K_{2}\subset F_{2}$, where $F_{j}=\mbox{Span}\{e_{i},i\in I_{j}\}$, for $j=1$, $2$ such that for any $x_{1}\in F_{1}$ and $x_{2}\in F_{2}$, for $x=x_{1}+x_{2}$, one has $g(x)=(1-s\|x_{1}\|_{K_{1}})^{\frac{1}{s}}_{+}1_{K_{2}}(x_{2})$ for $s\neq0$ and $g(x)=e^{-\|x_{1}\|_{K_{1}}}1_{K_{2}}(x_{2})$ for $s=0$.

\forget
Hence, the inequality in Conjecture \ref{conj:s-concave} can be simplified as follows:
\begin{conjecture}\label{conj:concavedimn}
Let $m\in\Z$ and $f:\R^{n}\to\R_{+}$ be an even concave function such that $0<\int_{\R^{n}}f<+\infty$. Then,
$$
P_{m}(f)=\int_{\R^{n}}f(x)^{m}dx\int_{\R^{n}}(\L f(y))^{m}dy\geq\frac{4^{n}}{(m+1)...(m+n)},
$$
where $\L f(y)=\inf_{x\in\R^{n}}\frac{\left(1-\langle x,y\rangle \right)_{+}}{f(x)}$.
\end{conjecture}
\forgotten

%The $1$-dimensional case of the previous theorem for even $s$-concave functions was proved for every $s\in\R$ (see \cite{SR}). 
In our third main result, we establish Conjecture \ref{comp-bipolar} in dimension $n=1$.
\begin{theorem}\label{comp-bipolar-one}
Let $s\in(-1,0)$ and let $g:\R\to\R_{+}$ be even $s$-concave such that $0<\int g<+\infty$. Then,
$$
\int_{\R}g(x)dx\ge(1+s)\int_{\R}\L_s\L_{s}g(y)dy,
$$
with equality for $g=\1_{[-1,1]}$.
\end{theorem}

\begin{remark}
Notice that, if Conjecture \ref{comp-bipolar} holds for any $s$-concave even function $g$ and Conjecture~\ref{mahler-s-F} holds for $\L_s g$, then multiplying equation \eqref{eq:bipolar} by $\int\L_s g$, using that $(\L_s g)^s\in\F$ and applying Conjecture~\ref{mahler-s-F} to $\L_sg$, then we deduce that Conjecture~\ref{mahler-s} holds for $g$.  
\end{remark}
From the above remark, the results in \cite{FGSZ} and Theorem \ref{comp-bipolar-one}, one deduces that Conjecture~\ref{mahler-s} holds for $n=1$ and $s\in(-1,0)$, which is the following theorem.
\begin{theorem}\label{th:mahler-n-1}
Let $s\in(-1,0)$ and $g:\R\to\R_{+}$ be even convex such that $0<\int g<+\infty$. Then,
$$
\int_{\R}g(x)dx\int_{\R}\L_{s}g(y)dy\geq\frac{4}{(1+s)^2},
$$
with equality if and only if $g=c\1_{[-a,a]}$, for some $c>0$ and $a>0$.
\end{theorem}

%we first give a shorter and simpler proof of the preceding theorem in the case where $n=1$ and $1/s\in\mathbb{N}$ and we will also prove the case where $n=1$ and $\frac{1}{s}<-1$.
%following conjecture in dimension $1$ for $s=\frac{1}{m}>0$ and $s=-\frac{1}{m}\in]-1,0[$. 
%\begin{conjecture}\label{conj:concavedim1}    Let $s>-1$ and $g:\R\to\R_{+}$ be an even $s$-concave function such that $0<\int_{\R}g<+\infty$. Then,\[\int_{\R}g(x)dx\int_{\R}\L_{s}g(y)dy\geq \frac{4}{1+s}.\]\end{conjecture}
%The main part of our paper concerns the $2$-dimensional version of the previous theorem for even $s$-concave functions for $s=\frac{1}{m}$ and $s=-\frac{1}{m}$ where $m\in\N^{*}$.
%Unfortunately, for the case where $s<0$, we shall restrain our result for lower semi-continuous $s$-concave functions admitting an asymptote in every non-zero direction.
%\begin{theorem}\label{thm:sneg}
%Let $m\in\N^{*}$, $s=-\frac{1}{m}$ and $g:\R^{2}\to\R_{+}$ be an even $s$-concave function admitting an asymptote in every non-zero direction such that $0<\int g<+\infty$. Then,
%\[
%\int_{\R^{2}}g(x)dx\int_{\R^{2}}\L_{s}g(y)dy\geq\frac{s^{2}}{1-2s}\frac{16}{(s-1)(2s-1)}.
%\]
%\end{theorem}

The proof's strategy of our main theorems in dimension $n=2$ relies on constructing an associated symmetric convex body in dimension $m+2$, where $m=1/|s|$, to which we apply the following proposition concerning Mahler's conjecture in higher dimensions proved in \cite{FHMRZ}.
\begin{proposition}\label{prop:Mahler} {\bf{[FHMRZ]}}
Let $K\subset\R^{n}$ be a centrally symmetric convex body that can be partitioned with hyperplanes $H_{1}$, $H_{2}$,..., $H_{n}$ into $2^{n}$ pieces of the same volume such that each section $K\cap H_{i}$ satisfies Mahler's conjecture and is partitioned into $2^{n-1}$ regions of the same $(n-1)$-dimensional volume by the remaining hyperplanes, then
\[
P(K)\geq\frac{4^{n}}{n!}.
\]
\end{proposition}
In our proofs, we shall need  the following notion of equipartition of functions.
\begin{definition}
Let $m\in\Z$. A function $f:\R^{n}\to\R_{+}$ is said to be $m$-equipartitioned if 
$$
\int_{\R^{n}_{\varepsilon}}f^m=\frac{1}{2^{n}}\int_{\R^{n}}f^{m}\quad\mbox{and}\quad\int_{\R_{\varepsilon}^{n}} f^{m-1}=\frac{1}{2^{n}}\int_{\R^{n}}f^{m-1},
$$
where $\forall\varepsilon\in\{-1,1\}^{n}$, $\R^{n}_{\varepsilon}=\{x\in\R^{n};\varepsilon_{i}x_{i}\geq 0,\forall i\in\{1,...,n\}\}$.   
\end{definition}

This paper is organized in the following way: In section 2, we present some general results and properties of the $s$-dual function $\L_{s}$ and we introduce the construction of convex bodies using $s$-concave functions. In section 3, we establish Conjecture \ref{mahler-s} in dimensions $n=1$ and $n=2$ for $\frac{1}{s}\in\N$. In section 4, we give a lower bound of the functional volume product applied to even $s$-concave functions in dimensions $n=1$ and $n=2$ for $-\frac{1}{s}\in\N$. 

\section{Construction and properties of associated convex bodies.}
Let $s>-\frac{1}{n}$ and $g:\R^{n}\to\R_{+}$ be an $s$-concave function. We discuss two cases according to the sign of $s$.

\subsection{Case $s>0$.} 
Let $m=\frac{1}{s}>0$. Assume that $0<\int g<+\infty$. Then there exists a function $f:\R^{n}\to\R_{+}$ concave on its support such that $g=f^{\frac{1}{s}}=f^{m}$ and
\[
\L_{s}g(y)=\L_{s}f^{\frac{1}{s}}(y)=\inf_{x\in\R^{n}}\frac{\left(1-s\langle x,y\rangle\right)_{+}^{\frac{1}{s}}}{f(x)^{\frac{1}{s}}}=\left(\inf_{x\in\R^{n}}\frac{\left( 1-\langle x,sy\rangle\right)_{+}}{f(x)}\right)^{\frac{1}{s}}=\left(\L f(sy)\right)^{\frac{1}{s}}=\left( \L f\left(\frac{y}{m}\right)\right)^{m},
\]
where $\L f(y)=\L_{1}f(y)=\inf_{x\in\R^{n}}\frac{\left(1-\langle x,y\rangle \right)_{+}}{f(x)}$. 
In addition, by the change of variables $y=mz$, we get
$$
\int_{\R^{n}}\L_{s}g(y)dy=m^{n}\int_{\R^{n}}(\L f(z))^{m}dz.
$$
Introducing the operator $P_{m}$ defined as follows
\[
P_{m}(f)=\int_{\R^{n}}f(x)^{m}dx\int_{\R^{n}}(\L f(y))^{m}dy,
\]
and changing variables, we get
\[
\int_{\R^{n}}g(x)dx\int_{\R^{n}}\L_{s}g(y)dy=m^{n}P_{m}(f).
\]
Thus, for $s>0$, Conjecture \ref{mahler-s} can equivalently be formulated in the following form: for any even function $f:\R^{n}\to\R_{+}$, concave on its support and integrable and for any $m>0$, one should have 
\begin{equation}\label{eq:conjPmf}
P_m(f)=\int_{\R^{n}}f(x)^{m}dx\int_{\R^{n}}(\L f(y))^{m}dy\ge P_m(\1_{[-1,1]^n})=\frac{4^n}{(m+1)\cdots(m+n)}.
\end{equation}
The operator $P_{m}$ is invariant under linear transformation. Indeed, let $T:\R^n\to\R^n$ be an invertible linear map. Then, putting $z=Tx$ we get,
\[
\int_{\R^{n}}(f\circ T(x))^{m}dx=\frac{1}{|\det T|}\int_{\R^{n}}f(z)^{m}dz.
\]
We also obtain, for every $y\in\R^n$,
\[
\L(f\circ T)(y)=\inf_{x\in\R^{n}}\frac{\left(1-\langle x,y\rangle\right)_{+}}{f(Tx)}=\inf_{z\in\R^{n}}\frac{\left(1-\langle T^{-1}z,y\rangle\right)_{+}}{f(z)}=\L f\left((T^{-1})^{*}y\right).
\]
Therefore, $\L(f\circ T)=(\L f)\circ (T^{-1})^*$. Hence, changing variables, we get
\begin{eqnarray*}
P_{m}(f\circ T)%&=&
%\int_{\R^{n}}(f\circ T(x))^{m}dx\int_{\R^{n}}(\L(f\circ T)(y))^{m}dy\\
&=&\int_{\R^{n}}(f(Tx))^{m}dx\int_{\R^{n}}(\L f((T^{-1})^{*}y))^{m}dy\\
&=&\frac{1}{|\det T|}\int_{\R^{n}}f(x)^{m}dx\times |\det T^{*}|\int_{\R^{n}}(\L f(y))^{m}dy=P_{m}(f).
\end{eqnarray*}
Let us introduce some tools that will help us in the proof of our main theorem in section 3. \\
Let $L$ be a symmetric convex body in $\R^{m}$. Define the set 
\[
K_{L,m}(f)=\{(x,t)\in\R^{n}\times\R^{m};\|t\|_{L}\leq f(x)\},
\]
where $\|.\|_{L}$ is defined for $y\in\R^{m}$ by $\|y\|_{L}=\inf\{\lambda>0; y\in \lambda L\}$. Variants of these sets and their volume products were considered in \cite{AKM, BF} but we establish their properties below for completeness. 
Since $f$ is even and concave, then $K_{L,m}(f)$ is a symmetric convex body in $\R^{m+n}$ and for every $m>0$ one has
\[
|K_{L,m}(f)|=|\{(x,t)\in\R^{n}\times\R^{m};\|t\|_{L}\leq f(x)\}|=\int_{\R^{n}}|f(x)L|_{m}dx=|L|\int_{\R^{n}}f(x)^mdx.
\]
Moreover, we have 
\begin{eqnarray*}
(K_{L,m}(f))^{\circ}&=&\{(y,s)\in\R^{n}\times\R^{m};\langle x,y\rangle+\langle s,t\rangle\leq 1,\quad\forall (x,t)\in K_{L,m}(f)\}\\
%&=&\{(y,s)\in\R^{n}\times\R^{m};\quad\langle x,y\rangle+\langle s,t\rangle\leq 1;\quad\forall (x,t) \mbox{ such that }\|t\|_{L}\leq f(x)\}\\
%&=&\{(y,s)\in\R^{n}\times\R^{m};\quad\langle s,t\rangle\leq 1-\langle x,y\rangle;\quad \forall (x,t) \mbox{ such that }\|t\|_{L}\leq f(x)\}\\
&=&\{(y,s)\in\R^{n}\times\R^{m}; \sup_{\{t\in\R^{m};\|t\|_{L}\leq f(x)\}}\langle s,t\rangle\leq 1-\langle x,y\rangle,\quad\forall x\in\R^n\}\\
&=&\left\{(y,s)\in\R^{n}\times\R^{m};f(x)\|s\|_{L^{\circ}}\leq (1-\langle x,y\rangle)_+,\quad\forall x\in\R^n\right\}\\
%&=&\left\{(y,s)\in\R^{n}\times\R^{m};\quad \|s\|_{L^{\circ}}\leq\frac{(1-\langle x,y\rangle)_+}{f(x)},\quad\forall x\in\R^n\right\}\\
&=&\left\{(y,s)\in\R^{n}\times\R^{m}; \|s\|_{L^{\circ}}\leq\L f(y)\right\}\\
&=&K_{L^{\circ},m}(\L f).
\end{eqnarray*}
Hence
\[
|K_{L,m}(f)^{\circ}|=|K_{L^{\circ},m}(\L f)|=|L^{\circ}|\int_{\R^{n}}(\L f(y))^{m}dy.
\]
Thus, one has 
\begin{eqnarray*}
P(K_{L,m}(f))&=&|K_{L,m}(f)||K_{L^{\circ},m}(\L f)|
=|L||L^{\circ}|\int_{\R^{n}}f(x)^mdx\int_{\R^{n}}(\L f(y))^{m}dy
=P(L)P_{m}(f).
\end{eqnarray*}
We need the following lemma describing the relation between the concave function $f$ and its dual.
\begin{lemma}\label{supp}
Let $f:\R^{n}\to\R_{+}$ be a concave function on its support. Then,
\[
\supp(\L f)=(\supp f)^{\circ}.
\]
\end{lemma}
\begin{proof}
Let $C=\supp(f)$ and $K(f)=\{(x,t)\in C\times\R; |t|\leq f(x)\}$. We know that $K(f)$ is convex in $\R^{n+1}$ and $K(f)^{\circ}=K(\L f)$.
%\begin{eqnarray*}
%K(f)^{\circ}&=&\{(y,s)\in C\times\R; \quad\langle x,y\rangle+\langle s,t\rangle\leq 1, \quad\forall (x,t)\in K(f)\}\\
%&=&\{(y,s)\in C\times\R; \quad\langle x,y\rangle+\langle s,t\rangle\leq 1, \quad\forall (x,t) \mbox{ such that } |t|\leq f(x)\}\\
%&=&\{(y,s)\in C\times\R; \quad\langle s,t\rangle\leq 1-\langle x,y\rangle, \quad\forall(x,t)\mbox{ such that } |t|\leq f(x)\}\\
%&=&\{(y,s)\in C\times\R; \sup_{\{t\in\R;|t|\leq f(x)\}}\langle s,t\rangle\leq 1-\langle x,y\rangle,\quad\forall x\in C\}\\
%&=&\left\{(y,s)\in C\times\R;\quad f(x)|s|\leq (1-\langle x,y\rangle)_+,\quad\forall x\in C\right\}\\
%&=&\left\{(y,s)\in C\times\R ;\quad |s|\leq\frac{(1-\langle x,y\rangle)_+}{f(x)},\quad\forall x\in C\right\}\\
%&=&\left\{(y,s)\in C\times\R ;\quad |s|\leq\L f(y)\right\}\\
%&=&K(\L f).
%\end{eqnarray*}
Since the operation of duality transforms sections of a given convex body into projections of its polar and 
\[
K(f)\cap e_{n+1}^{\perp}=C=P_{e_{n+1}^{\perp}}(K(f)),
\] 
we have
\[
K(\L f)\cap e_{n+1}^{\perp}=K(f)^{\circ}\cap e_{n+1}^{\perp}=\left(P_{e_{n+1}^{\perp}}(K(f))\right)^{\circ}.
\]
Thus, 
\[
\supp\L f=C^{\circ}=(\supp f)^{\circ}.
\]
\end{proof}

\subsection{Case $-\frac{1}{n}<s<0$.} 
Let $m=-\frac{1}{s}>n$. Then there exists a convex function $f:\R^{n}\to(0,+\infty]$ such that $f=g^{s}$. Assume that $0<\int g<+\infty$, one has $\lim_{t\to+\infty}f(tx)=+\infty$, for any $x\neq 0$. Moreover, since $s=-1/m$ and $g=f^{-m}$,
\begin{eqnarray*}
\L_{s}g(y)=\inf_{x\in\R^{n}}\frac{\left(1-s\langle x,y\rangle\right)^{\frac{1}{s}}_{+}}{g(x)}
%=\left(\inf_{x\in\R^{n}}\frac{\left(1+\langle x,y\rangle\right)^{-1}_{+}}{g(x)^{-s}}\right)^{\frac{-1}{s}}
=\left(
\inf_{x\in\R^{n}}\frac{(1+\langle x,\frac{y}{m}\rangle)^{-1}_{+}}{f^{-1}(x)}\right)^{m} 
=\left(\L_{-1}(f^{-1})\left(\frac{y}{m}\right)\right)^{m}.
\end{eqnarray*}
Let us define the function $\M f$, for $y\in\R^{n}$, by
\[
\M f(y)=(\L_{-1}(f^{-1})(y))^{-1}=\sup_{x\in\R^{n}}\frac{1+\langle x,y\rangle}{f(x)}.
\]
It was proved in \cite{FGSZ} that, with these hypotheses, the function $\M f$ belongs to $\F$ (see definition \ref{def:F}).
Using the change of variables $y=mz$, we get
\[
\int_{\R^{n}}\L_{s}g(y)dy=m^{n}\int_{\R^{n}}\left(\L_{-1}(f^{-1})(z)\right)^{m}dz=m^{n}\int_{\R^{n}}\frac{1}{(\M f(z))^{m}}dz.
\]
Consider the operator $Q_{m}$ defined as follows
\[
Q_{m}(f)=\int_{\R^{n}}\frac{1}{f(x)^m}dx\int_{\R^{n}}\frac{1}{(\M f(y))^{m}}dy,
\]
one has
\begin{eqnarray*}
\int_{\R^{n}}g(x)dx\int_{\R^{n}}\L_{s}g(y)dy
%&=&m^{n}\int_{\R^{n}}(f(x))^{\frac{1}{s}}dx\int_{\R^{n}}\frac{1}{(\M f(z))^{m}}dz\\
=m^{n}\int_{\R^{m}}\frac{1}{f(x)^m}dx\int_{\R^{n}}\frac{1}{(\M f(z))^{m}}dz=m^{n}Q_{m}(f).
\end{eqnarray*}
Thus, for $s<0$, Conjecture \ref{mahler-s} can equivalently be formulated in the following form: for any even function $f:\R^{n}\to(0,+\infty]$, convex, such that $\lim_{t\to+\infty}f(tx)=+\infty$, for any $x\neq 0$ and for any $m>n$, one should have 
\begin{equation}\label{eq:conjQmf}
Q_m(f)=\int_{\R^{n}}\frac{1}{f(x)^{m}}dx\int_{\R^{n}}\frac{1}{(\M f(y))^{m}}dy\ge Q_m(1+I_{[-1,1]^n})=\frac{4^n}{(m-1)\cdots(m-n)},
\end{equation}
where, for any $K\subset\R^n$, the function $I_K$ is defined by $I_K(x)=0$ for $x\in K$ and $I_K(x)=+\infty$ for $x\notin K$. 

On the other hand, if $f\in\F$, then it was proved in \cite{FGSZ} that for any $u\in S^{n-1}$ there exists a continuous positive function $a$ such that $f(tu)\ge a(u)t$ for every $t>0$. By compactness of the sphere and continuity of $a$ there exists $c>0$ such that $f(tu)\ge ct$ for every $t>0$. We deduce that $f(x)\ge c\|x\|_2$ for every $x\in \R^n$. Moreover, since $f$ is positive, continuous and goes to infinity at infinity, there exists $c'>0$ such that $f(x)\ge c'$ for every $x\in\R^n$. We conclude that $f(x)\ge\max(c\|x\|_2,c')$. Thus, 
\begin{equation}\label{eq:gint}
g(x)=f(x)^\frac{1}{s}\le \max(c\|x\|_2,c')^\frac{1}{s},
\end{equation}
which is integrable since $-\frac{1}{n}<s<0$. Therefore, for any $f\in\F$, the function $g=f^\frac{1}{s}$ is integrable.

Same as before, Conjecture \ref{mahler-s-F} can be reformulated in the form: for $f\in\F$ even and $m>n$,
\begin{equation}\label{eq:conjQmf-F}
Q_m(f)\ge Q_m(\max(1,\|x\|_\infty))=\frac{m}{m-n}\times\frac{4^n}{(m-1)\cdots(m-n)}.
\end{equation}
And Conjecture \ref{comp-bipolar} reads as: for any even function $f:\R^{n}\to(0,+\infty]$, convex such that $\lim_{t\to+\infty}f(tx)=+\infty$, for any $x\neq 0$ and for any $m>n$, 
\begin{equation}\label{eq:conj-bipolar-f}
\int_{\R^n}\frac{1}{f(x)^m}dx\ge \frac{m-n}{m}\int_{\R^n}\frac{1}{(\M\M f(x))^m}dx,
\end{equation}
with equality for $f=1+I_{[-1,1]^n}$.

Similarly as before, one may easily prove that the operator $Q_{m}$ is invariant under linear transformation.
Let $K\subset\R^{m}$ be a symmetric convex body. For any even convex function $f:\R^{n}\to\R_{+}$ such that $\lim_{x\to\infty}f(tx)=+\infty$, for any $x\neq 0$, we define the function 
\[
\begin{array}{ccccl}
\pi_m(f) & : & \R^{n}\times\R^{m} & \to & \R_{+}\cup\{+\infty\} \\
 & & (x,s) & \mapsto & \pi_m(f)(x,s)=\begin{cases}
  \|s\|_{K} f(\frac{x}{\|s\|_{K}}) & \text{if $s\neq 0$}. \\
  \lim_{s\to 0}\|s\|_{K}f(\frac{x}{\|s\|_{K}}) & \text{if $s=0$}.
\end{cases} \\
\end{array}
\]
We denote the domain of $f$ by $\dom(f)=\{x\in\R^{n}; f(x)<+\infty\}$. Then 
\[
\dom(\pi_m(f))=\{(x,s)\in\R^{n}\times\R^{m}; \pi_m(f)(x,s)<+\infty\}=\{(x,s)\in\R^{n}\times\R^{m}; x\in \|s\|_{K}\dom(f)\}.
\]
Consider the following convex body 
\[
C^{K}_{m}(f)=\{(x,s)\in\R^{n}\times\R^{m}; \pi_m(f)(x,s)\leq 1\}=\overline{\left\{(x,s)\in\R^{n}\times\R^{m}\setminus\{0\}; \|s\|_{K}f\left(\frac{x}{\|s\|_{K}}\right)\leq 1\right\}}.
\]
One has,
\begin{eqnarray*}
 (C^{K}_{m}(f))^{\circ}&=&\{(y,t)\in\R^{n}\times\R^{m}; \langle (x,s),(y,t)\rangle\leq 1, \forall (x,s)\in C^{K}_{m}(f)\}\\
 &=&\left\{(y,t)\in\R^{n}\times\R^{m}; \langle x,y\rangle + \langle s,t\rangle\leq 1, \forall (x,s); \|s\|_{K}f\left(\frac{x}{\|s\|_{K}}\right)\leq 1\right\}.
\end{eqnarray*}
Let $z=\frac{x}{\|s\|_{K}}$, then,
\begin{eqnarray*}
    (C^{K}_{m}(f))^{\circ}&=&\left\{(y,t)\in\R^{n}\times\R^{m}; \|s\|_{K}\langle z,y\rangle+\langle s,t\rangle\leq 1, \forall (z,s), f(z)\leq\frac{1}{\|s\|_{K}}\right\}\\
    &=&\left\{(y,t)\in\R^{n}\times\R^{m};\sup_{\left\{s;\|s\|_{K}\leq1/f(z)\right\}} \|s\|_{K}\langle z,y\rangle +\langle s,t\rangle\leq 1, \forall z\right\}\\
    &=&\left\{(y,t)\in\R^{n}\times\R^{m};\frac{\langle z,y\rangle+\|t\|_{K^\circ}}{f(z)}\leq 1, \forall z\right\}\\
    &=&\left\{(y,t)\in\R^{n}\times\R^{m}; \sup_{z\in\R^{n}}\frac{\langle z,y\rangle+\|t\|_{K^{\circ}}}{f(z)}\leq 1\right\}\\
    %&=&\left\{(y,t)\in\R^{n}\times\R^{m}; \|t\|_{K^\circ}\sup_{z\in\R^{n}}\frac{\left\langle z,\frac{y}{\|t\|_{K^\circ}}\right\rangle+1}{f(z)}\leq 1\right\}\\
    &=&\overline{\left\{(y,t)\in\R^{n}\times\R^{m}\setminus\{0\}; \|t\|_{K^\circ}\M f\left(\frac{y}{\|t\|_{K^\circ}}\right)\leq 1\right\}}\\
    &=& C^{K^\circ}_{m}(\M f).
\end{eqnarray*}
Let us compute the volume of the convex body $C_{m}^{K}(f)$.
\begin{eqnarray*}
    \left|C_{m}^{K}(f)\right|&=&\left|\left\{(x,s)\in\R^{n}\times\R^{m}\setminus\{0\};\|s\|_{K}f\left(\frac{x}{\|s\|_{K}}\right)\leq 1\right\}\right|\\
    &=&\int_{\R^{m}}\left|\left\{x\in\R^{n}; \|s\|_{K}f\left(\frac{x}{\|s\|_{K}}\right)\leq 1\right\}\right|_{n}ds\\
   % &=&\int_{\R^{m}}\left|\left\{z\|s\|_{K}\in\R^{n}; f(z)\leq\frac{1}{\|s\|_{K}}\right\}\right|_{n}\quad ds\\
    &=&\int_{\R^{m}}\|s\|_{K}^{n}\left|\left\{z\in\R^{n};f(z)\leq\frac{1}{\|s\|_{K}}\right\}\right|_{n}ds=\int_{\R^{m}}\varphi(\|v\|_K)dv,
\end{eqnarray*}
where $\f(t)=t^{n}\left|\left\{z\in\R^{n};f(z)\leq\frac{1}{t}\right\}\right|_{n}.$
\forget 
We shall use the following lemma.
\begin{lemma} 
Let $m>0$. Let $K$ be a symmetric convex body in $\R^{m}$ and $\varphi:\R\to\R_{+}$ a function such that $\lim_{t\to+\infty}t^{m}\varphi(t)=0$. Then,
\[
\int_{\R^{m}}\varphi(\|s\|_{K})ds=|K|\int_{0}^{+\infty}mt^{m-1}\varphi(t)dt.
\]
\end{lemma}
%\begin{proof}
\forgotten
Using Fubini and a change of variables, one has
\begin{eqnarray*}
\int_{\R^{m}}\f(\|v\|_{K})dv%&=&\int_{\R^{m}}\left[-\f(t)\right]^{+\infty}_{\|s\|_{K}}ds=\int_{\R^{m}}\int_{\|s\|_{K}}^{+\infty}(-\f'(t)) dt ds
=\int_{0}^{+\infty}(-\f'(t))|\{v\in\R^{m};\|v\|_{K}\leq t\}|dt
= |K|\int_{0}^{+\infty} t^{m}(-\f'(t))dt
%=|K|\left(\left[t^{m}(-\f(t))\right]^{+\infty}_{0}+\int_{0}^{+\infty}mt^{m-1}\f(t)dt\right)\\
=m|K|\int_{0}^{+\infty}t^{m-1}\f(t)dt.
\end{eqnarray*}
%\end{proof}
Hence, we get
\begin{eqnarray*}
    |C_{m}^{K}(f)|%&=&\int_{\R^{m}}\|s\|_{K}^{n}\left|\left\{z\in\R^{n};f(z)\leq\frac{1}{\|s\|_{K}}\right\}\right|_{n}\quad ds\\
    &=& m |K|\int_{0}^{+\infty}t^{m-1}t^{n}\left|\left\{z\in\R^{n};f(z)\leq\frac{1}{t}\right\}\right|_{n} dt\\
    &=&m |K|\int_{0}^{+\infty}\frac{1}{u^{m+n+1}}\left|\left\{z\in\R^{n}; f(z)\leq u\right\}\right| du\\
    %&=&m|K|\int_{0}^{+\infty}\frac{1}{u^{m+n+1}}\int_{\R^{n}}\1_{\{z\in\R^{n}; f(z)\leq u\}}dz du\\
    &=&m|K|\int_{\R^{n}}\int_{f(z)}^{+\infty}\frac{du}{u^{m+n+1}}dz\\
    %&=&\frac{m}{m+n}|K|\int_{\R^{n}}\left[\frac{-1}{u^{m+n}}\right]^{+\infty}_{f(z)}dz\\
    &=&\frac{m}{m+n}|K|\int_{\R^{n}}\frac{dz}{f(z)^{m+n}}.
\end{eqnarray*}
Using that $C_{m}^{K}(f)^{\circ}=C_{m}^{K^\circ}(\M f)$, we get
\[
|C_{m}^{K}(f)^{\circ}|=\frac{m}{m+n}|K^{\circ}|\int_{\R^{n}}\frac{dz}{(\M f(z))^{m+n}}.
\]
Thus, we obtain for any $m>0$
\begin{equation}\label{eq:Mf}
\int_{\R^{n}}\frac{1}{f(x)^{m+n}}dx\int_{\R^{n}}\frac{1}{(\M f(y))^{m+n}}dy=\left(\frac{m+n}{m}\right)^{2}\frac{1}{|K||K^{\circ}|}|C_{m}^{K}(f)||C_{m}^{K}(f)^{\circ}|.
\end{equation}

\section{Proof of the inequalities in dimension 1 and 2 for $s>0$.}

In this section, we start by giving a monotonicity result, which implies Conjecture~\ref{mahler-s} in dimension $1$ where $\frac{1}{s}\in\N$. Recall that Conjecture~\ref{mahler-s} in dimension $1$ was proved in \cite[page 17]{FM10} for any $s>0$. We present here a simpler proof in the case where $\frac{1}{s}\in\N$. Then, we present a proof of Theorem \ref{thm:spos}, which is Mahler's conjecture for $s$-concave functions in the case $n=2$ and $\frac{1}{s}\in\N$. 

Let $m=\frac{1}{s}$ and $g:\R^{n}\to\R_{+}$ be an even $s$-concave function, then there exists an even function $f:\R^{n}\to\R$, concave  on its support, such that $g=f^{m}$. Thus, in dimension $n=1$ and for $1/s\in\N$, Conjecture~\ref{mahler-s} is a consequence of the following theorem.

\begin{theorem}\label{dim1}
Let $m\in\N$, $a>0$, $f:\R\to\R_{+}$ be an even concave function on its support $C=[-a,a]$ and
\[
P_{m}(f)=\int_{\R}f(x)^{m}dx\int_{\R}(\L f(y))^{m}dy.
\]
Then the sequence $((m+1)^2P_m(f)-4m)_m$ is non-decreasing. Hence, one has $P_m(f)\geq\frac{4}{m+1}$.
Moreover, if $P_m(f)=\frac{4}{m+1}$ then either $f$ is constant on its support $[-a,a]$ or there is $c>0$ such that $f(x)=c\left(1-\frac{|x|}{a}\right)_+$.
\end{theorem}
\begin{proof}
First, using Lemma \ref{supp}, we know that $\supp(\L f)=C^{\circ}=[-\frac{1}{a},\frac{1}{a}]$. 
%Since $f$ and $\L f$ are even, it suffices to prove that 
%\[\int_{0}^{a}f(x)^{m}dx\int_{0}^{\frac{1}{a}}(\L f(y))^{m}dy\geq\frac{1}{m+1}.\]
Set 
\[
 I_{m}=(m+1)\int_{0}^{a}f(x)^{m}dx\quad\mbox{and}\quad J_{m}=(m+1)\int_{0}^{\frac{1}{a}}(\L f(y))^{m}dy.
\]
The monotonicity that we want to prove reduces to $I_{m}J_{m}\geq m+1$, for every $m\in\N$. The inequality holds for $m=0$ (there is equality). Let $m\in\N^*$. Being concave on $[-a,a]$, the function $f$ is left differentiable on $(-a,a]$. By concavity, denoting by $f'(t)$ the left derivative of $f$ at $t$, we have for all $t\in(-a,a]$ and $x\in[-a,a]$
 \begin{eqnarray}\label{eq:f-f'-dim1}
 f(x)\leq f(t)+(x-t)f'(t).
 \end{eqnarray}
multiplying both sides of the inequality by $mf(t)^{m-1}$ and integrating on $[0,a]$  gives
 \[
f(x)I_{m-1}= f(x)\int_{0}^{a}mf(t)^{m-1}dt\leq \int_{0}^{a}mf(t)^mdt+\int_{0}^{a}m(x-t)f'(t)f(t)^{m-1}dt.
 \]
 Integration by parts gives that for every $x\in[-a,a]$
 \begin{eqnarray*}
 \int_{0}^{a}m(x-t)f'(t)f(t)^{m-1}dt
 %&=&\left[ (x-t)\frac{f(t)^m}{m}\right]^{a}_{0}+\int_{0}^{a}\frac{f(t)^m}{m}dt\\
 =(x-a)f(a)^{m}-xf(0)^{m}+\int_{0}^{a}f(t)^mdt.
 \end{eqnarray*}
Using that $(x-a)f(a)^{m}\le 0$, we get that for all $x\in[-a,a]$,
\begin{equation}\label{ineq-Im}
f(x)I_{m-1}\leq I_m-xf(0)^{m}.
\end{equation}
Hence, one has for all $x\in[-a,a]$
\[
\frac{I_{m-1}}{I_{m}}\leq\frac{1}{f(x)}\left(1-\frac{xf(0)^{m}}{I_{m}}\right).
\]
By taking the infinimum over all $x\in[-a,a]$, and denoting $y_0=\frac{f(0)^{m}}{I_{m}}$ we get
\[
\frac{I_{m-1}}{I_{m}}\leq\L f\left(y_0\right).
\]
Since $\L f$ is concave, we can apply the result to $\L f$ and, denoting $x_0=\frac{(\L f(0))^{m}}{J_{m}}$, we get
\[
\frac{J_{m-1}}{J_{m}}\leq f\left(x_0\right).
\]
On the other hand, the definition of $\L f$ implies that for all $x\in\supp f$, $y\in\supp\L f$
\[
f(x)\L f(y)\leq (1-xy)_{+}.
\]
Using \eqref{ineq-Im} for $x=0$ gives that $I_m\ge f(0)I_{m-1}$, thus $I_m\ge f(0)^ma$. In the same way, one has $J_m\ge \L f(0)^ma^{-1}$. Multiplying these inequalities and using that $f$ is even and concave, we get $\L f(0)=\max(f)^{-1}=f(0)^{-1}$, we conclude that $I_mJ_m\ge1$, which implies that $x_0y_0\le1$.
Hence, one has
\begin{eqnarray}\label{eq:Im-Jm-x_0}
\frac{I_{m-1}J_{m-1}}{I_{m}J_{m}}\leq f\left(x_0\right)\L f\left(y_0\right)\leq 1-x_0y_0= 1-\frac{1}{I_{m}J_{m}}.
\end{eqnarray}
Finally we get $I_{m-1}J_{m-1}\leq I_{m}J_{m}-1$. Hence the sequence $(I_mJ_m-m)_m$ is non-decreasing. Since $I_{m}J_{m}=\frac{(m+1)^2}{4}P_m(f)$, 
we deduce that the sequence $u_m=(m+1)^2P_m(f)-4m$ is non-decreasing. It follows that $u_m\ge u_0=P_0(f)=4$, hence $P_m(f)\ge\frac{4}{m+1}$, for any integer $m$.

Assume now that there is equality, for some integer $m$: $P_m(f)=\frac{4}{m+1}$. Then, there is equality in \eqref{eq:Im-Jm-x_0}. It follows that there is equality in \eqref{ineq-Im} for $x=x_0$ and also $(x_0-a)f(a)^{m}=0$. Moreover, there is also equality in the corresponding inequality for $\L f$ instead of $f$ and $y_0$ instead of $x_0$, thus $(y_0-\frac{1}{a})\L f(1/a)^m=0$.
In addition, there is equality in \eqref{eq:f-f'-dim1} for $x=x_0$ and for almost any $t\in[-a,a]$: 
\[
f(t)= f(x_0)+(t-x_0)f'(t).
\]
On the other hand, one has $f(t)=f(x_0)+\int_{x_0}^tf'(u)du$. Since the left derivative $f'$ is non-increasing, we deduce that $f'$ is constant on $[0,x_0)$ and on $(x_0,a]$, which means that $f$ is affine on on $[0,x_0)$ and on $(x_0,a]$. In the same way, $\L f$ is affine on $[0,y_0)$ and on $(y_0,1/a]$.

If $x_0=a$, then $f$ is affine on $[0,a]$ and $J_m=\L f(0)^m/a$. But, by concavity, one has for every $y\in[0,1/a]$
\[
\L f(y)\ge\L f(0)(1-ay).
\]
Raising this inequality to the power $m$ and integrating, one gets that $J_m\ge\L f(0)^m/a$. Since there is equality we deduce that $\L f(y)=\L f(0)(1-ay)$,  for every $y\in[0,1/a]$. Thus $f=f(0)\1_{[-a,a]}$. 

If $x_0\neq a$, then $f(a)=0$. Thus there exists $0\le b\le1/a$ and $c>0$ such that, for $t\ge0$,
\[
f(t)=c(1-bt)\1_{[0,x_0]}+c(1-bx_0)\frac{a-x}{a-x_0}\1_{(x_0,a]}.
\]
This implies that, for any $y\ge0$, 
\[
\L f(y)=\frac{1}{c}\left(\1_{[0,b]}(y)+\frac{1-x_0y}{1-bx_0}\1_{(b,\frac{1}{a}]}(y) \right).
\]
Since $x_0\neq a$, one has $\L f(1/a)\neq0$, hence $y_0=1/a$. This gives $I_m=af(0)^m$ and this implies as before that $f(x)=f(0)(1-\frac{x}{a})$, for all $x\in[0,a]$. 
\forget
A simple computation gives 
\[
\int_0^{+\infty}f^m=\frac{c^m}{(m+1)b}\left(1-\left(1-ab\right)^{m+1} \right)\quad\hbox{and}\quad 
\int_0^{+\infty}(\L f)^m=\frac{1}{c^m}\left(ab+\frac{1}{m+1}(1-ab)\right).
\]
Hence, denoting $x=1-ab\in[0,1]$, we obtain that 
\[
P_m(f)=\frac{4(m+1-mx)}{(m+1)^2}\times\frac{1-x^{m+1}}{1-x}.
\]
Since we assumed that $f$ satisfies the equality case in the inequality $P_m(f)\ge4/(m+1)$, it means that $x\in[0,1]$ satisfies the equality case in the inequality 
\[
\frac{1}{1+x+\cdots +x^m}=\frac{1-x}{1-x^{m+1}}\le 1-\frac{m}{m+1}x.
\]
But from the strict convexity of the left hand side and the fact that there is equality for $x=0$ and $x=1$, the equality cases can only occur for $x=0$ or $x=1$. These cases correspond to $b=1/a$ or $b=0$ and this establishes the claimed equality cases.   
\forgotten 

\end{proof}
Let us prove the 2-dimensional case of our inequality for $m$ being an integer. 
\begin{theorem}\label{dim2}
Let $m\in\N$ and $f:\R^{2}\to\R_{+}$ be an even function, concave on its support, such that $0<\int f<+\infty$. Then,
$$
P_{m}(f)\geq \frac{16}{(m+1)(m+2)}.
$$
Moreover, there is equality for $f=\1_{[-1,1]^2}$.
\end{theorem}

As noticed before, this theorem may be reformulated in the following form: for any integer $m>0$ and any even function $g:\R^2\to\R_+$ that is $\frac{1}{m}$-concave, integrable with $\int_{\R^{2}}g>0$ one has 
\begin{eqnarray}\label{eq:m-concave}
\int_{\R^{2}}g\int_{\R^{2}}\L_{\frac{1}{m}}g\geq\frac{16m^2}{(m+1)(m+2)}. 
\end{eqnarray}
The following corollary proves that letting $m$ go to infinity we get a new proof of Mahler conjecture for even log-concave functions, which was our previous main result in \cite{FN}.

\begin{corollary}\label{cor:FN23}
Let $f:\R^{2}\to\R$ be an even log-concave function such that $0<\int f<+\infty$, then 
\begin{eqnarray}\label{eq:FN}
\int_{\R^{2}}f(x)dx\int_{\R^{2}}f^{\circ}(y)dy\geq 16,
\end{eqnarray}
where $f^{\circ}(y)=\inf_{x\in\R^{n}}\frac{e^{-\langle x,y\rangle}}{f(x)}
$ is the polar function of $f$.

\end{corollary}

\begin{proof}
Let $f$ be an even log-concave function and consider the function $f_{m}$ defined as follows
\[
f_{m}(x)=\left(1+\frac{\log f(x)}{m}\right)^{m}_{+}.
\]
First, notice that the log-concavity of $f$ implies the $\frac{1}{m}$-concavity of $f_{m}$. Since for all $x$, $e^{-x}\ge 1-x $, we get
\[
\left(1-\frac{1}{m}\langle x,y\rangle\right)^{m}\le e^{-\langle x,y\rangle}.
\]
Thus, $\L_{\frac{1}{m}}f\leq\L_{0}f=f^{\circ}$. In addition, it is easy to see that  $f_{m}\leq f$ for any $m>0$, and since a log-concave function
is continuous on its support one has $\lim_{m\to\infty}f_{m}=f$ locally uniformly on $\R^{2}$. Hence, 
%by the dominated convergence theorem, 
we can conclude that 
$\lim_{m\to\infty}\int_{\R^{2}}f_{m}=\int_{\R^{2}}f$.
%Let us recall that for a log-concave function $f$ we define its
%dual (or polar) as
%\[
%f^{\circ}(y)=\inf_{x\in\R^{n}}\frac{e^{-\langle x,y\rangle}}{f(x)}.
%\]
The fact that $\L_{\frac{1}{m}} f\leq f^{\circ}$ implies that $\L_{\frac{1}{m}}f_{m}\leq f_{m}^{\circ}$. Moreover, one has
%In addition, since $f_{m}\leq f$, then $\L_{m}f\leq \L_{m}f_{m}$. 
\[
\frac{16m^2}{(m+1)(m+2)}\leq \int_{\R^{2}}f_{m}\int_{\R^{2}}\L_{\frac{1}{m}}f_{m}\leq\int_{\R^{2}}f_{m}\int_{\R^{2}}f_{m}^{\circ}.
\]
On the other hand, since the functions $f_{m}$ are log-concave and $\lim_{m\to\infty}f_{m}=f$, then using lemma 3.2 in \cite{AKM}, we know that $\lim_{m\to\infty}f_{m}^{\circ}=f^{\circ}$ locally uniformly on the interior of the support of $f^{\circ}$. Applying the same lemma again on the log-concave function $f_{m}^{\circ}$, we get that $\lim_{m\to\infty}\int_{\R^{2}}f_{m}^{\circ}=\int_{\R^{2}}f^{\circ}$.
Thus, we conclude that
\[
\int_{\R^{2}}f\int_{\R^{2}}f^{\circ}=\lim_{m\to\infty}\int_{\R^{2}}f_{m}\int_{\R^{2}}f^{\circ}_{m}\geq 16. 
\]
\end{proof}
\begin{remark}
Notice that we don't get the equality case in this way. Recall that, in \cite{FN}, it was proved that there is equality in \eqref{eq:FN} if and only if there exists two Hanner polytopes  $K_1\subset F_1$ and $K_2\subset F_2$, where $F_1$ and $F_2$ are two supplementary subspaces in $\R^{2}$, with $0\le\dim(F_i)\le2$, such that for all $(x_1,x_{2})\in F_1\times F_{2}$
\[
f(x_{1},x_{2})=e^{-\|x_{1}\|_{K_{1}}}\1_{K_{2}}(x_{2}).
\]
It would be natural to conjecture that a concave function $f$ on its support satisfies the equality case in \eqref{eq:m-concave} if and only if 
\[
f(x_{1},x_{2})=(1-\|x_{1}\|_{K_{1}})_+\1_{K_{2}}(x_{2}),
\]
with $K_1, K_2, F_1, F_2$ as before. It is easy to see that such functions satisfy the equality case. Nevertheless, clearly, the function $f(x_1,x_2)=(1-|x_1|)_+\1_{B_1^2}(x_1,x_2)$ satisfies also the equality case but cannot be written in this form. 
\end{remark}
For the proof of theorem \ref{dim2}, let us start by proving that every even concave function $f:\R^{2}\to\R_{+}$ can be $m$-equipartitioned by the canonical basis of $\R^{2}$.
\begin{lemma}\label{lem:m-equip}
Let $m>0$ and $f:\R^{2}\to\R_{+}$ be an even concave function such that $0<\int_{\R^{2}}f<+\infty$. Then there exists a linear invertible map $T:\R^{2}\to\R^{2}$ such that the function $f\circ T$ is $m$-equipartitioned. 
\end{lemma}

\begin{proof}
The proof follows the proof of Lemma 4.2 in our previous work \cite{FN} but we reproduce it for completeness.
For any $u\in S^1$, let $C(u)\subset S^1$ be the open half-circle delimited by $u$ and $-u$ containing the vectors $v$ which are after $u$ with respect to the counterclockwise orientation of $S^1$. For $v\in C(u)$, let  $C(u,v)=\R_+u+\R_+v$ be the cone generated by $u$ and $v$ and define 
$$
f_u(v)=\int_{C(u,v)}f(x)^mdx.
$$
The map $f_u$ is continuous and increasing on $C(u)$, $f_u(u)=0$ and, since $f$ is even,  $f_u(v)\longrightarrow\frac{1}{2}\int_{\R^{2}}f^{m}(x)dx$ when $v\longrightarrow -u$. Thus, there exists a unique $v(u)\in C(u)$ such that 
$$
f_u(v(u))=\int_{C(u,v(u))}f^{m}(x)dx=\frac{1}{4}\int_{\R^2}f^{m}(x)dx.
$$
Notice that $v:S^1\to S^1$ is continuous and, since $f$ is even, one has 
$$
\int_{C(v(u),-u)}f^{m}(x)dx=\frac{1}{4}\int_{\R^{2}}f^{m}(x)dx,
$$
thus $v(v(u))=-u$ for any $u\in S^1$.
For $u\in S^1$, let  
$$
g(u)=\int_{C(u,v(u))}f^{m-1}(x)dx-\frac{1}{4}\int_{\R^{2}}f^{m-1}(x)dx.
$$
Then, $g$ is continuous on $S^1$ and, since $f$ is even, 
\begin{eqnarray*}
g(v(u))&=&\int_{C(v(u),-u)}f^{m-1}(x)dx-\frac{1}{4}\int_{\R^{2}}f^{m-1}(x)dx\\
&=&\frac{1}{2}\int_{\R^{2}}f^{m-1}(x)dx-\int_{C(u,v(u))}f^{m-1}(x)dx-\frac{1}{4}\int_{\R^{2}}f^{m-1}(x)dx\\
&=&\frac{1}{4}\int_{\R^{2}}f^{m-1}(x)dx-\int_{C(u,v(u))}f^{m-1}(x)dx\\
&=&-g(u).
\end{eqnarray*}
Hence, by the intermediate value theorem, there exists $u\in S^1$ such that $g(u)=0$, thus 
\[\int_{C(u,v(u))}f^{m-1}=\frac{1}{4}\int_{\R^{2}}f^{m-1}\quad\hbox{and}\quad \int_{C(u,v(u))}f^{m}=\frac{1}{4}\int_{\R^{2}}f^{m}.
\]
    
\end{proof}
%Let $f:\R^{n}\to\R_{+}$ be an even concave function such that $0<\int_{\R^{n}}f(x)dx<+\infty$ and 
Let us proceed to the proof of our main theorem.
\begin{proof}[\bf{Proof of Theorem \ref{dim2}}] Let $f:\R^{2}\to\R_{+}$ be an even concave function such that $0<\int_{\R^{2}}f<+\infty$. Let $L=B_{\infty}^{m}$. For simplicity, the convex body $K_{B_\infty^m,m}(f)$ introduced in section 2 will be denoted in this case by $K_{\infty,m}(f)$ and its dual $K_{B_1^m,m}(\L f)$ will be denoted by $K_{1,m}(\L f)$. Thus, we know that
\[
P(K_{\infty,m}(f))=P(B_{\infty}^{m})P_{m}(f)=\frac{4^{m}}{m!}P_{m}(f).
\]
The proof of the theorem concludes if $K_{\infty,m}$ verifies Mahler's conjecture in dimension $m+2$, i.e. $P(K_{\infty,m})\geq\frac{4^{m+2}}{(m+2)!}$. For that, it is enough to prove that the convex body $K_{\infty,m}$ verifies the hypotheses of Proposition \ref{prop:Mahler}. We proceed by induction on $m\in\N$. Set, for any positive integer $k$ and for any  $\varepsilon\in\{-1,1\}^{k}$, $\R^{k}_{\varepsilon}=\{x\in\R^{k};\varepsilon_{i}x_{i}\geq 0,\forall i\in\{1,...,k\}\}$.
\begin{itemize}
\item Let us start by proving that $K_{\infty,m}(f)$ can be partitioned into $2^{m+2}$ pieces of the same volume. Since $f$ is $m$-equipartitioned, we get
\begin{eqnarray*}
    \left|K_{\infty,m}(f)\cap\R^{m+2}_{+}\right|&=&\left|\{(x,t)\in\R^{2}_{+}\times\R_{+}^{m};\|t\|_{\infty}\leq f(x)\}\right|
    =\int_{\R^{2}_{+}}f(x)^{m}dx|B_{\infty}^{m}\cap\R_{+}^{m}|\\
    &=&\int_{\R^{2}_{+}}f(x)^{m}dx=\frac{1}{2^{2}|B_{\infty}^{m}|}\int_{\R^{2}}f(x)^{m}|B_{\infty}^{m}|dx
    =\frac{1}{2^{m+2}}|K_{\infty,m}(f)|.
\end{eqnarray*}
Using the fact that $\int_{\R^{2}_{+}}f(x)^{m}dx=\int_{\R_{+}\times\R_{-}}f(x)^{m}dx$, we get, 
\[
|K_{\infty,m}(f)\cap\R_{\e}^{m+2}|=\frac{1}{2^{m+2}}|K_{\infty,m}(f)|\quad\forall\e\in\{-1,1\}^{m+2}.
\]
\item Our next step is to prove that $K_{\infty,m}(f)\cap e_{i}^{\perp}$ verifies Mahler's conjecture for all $i\in\{1,...,m+2\}$.\\
For $i=1$ or $i=2$, the same argument holds so let us assume that $i=1$. One has
\[
K_{\infty,m}(f)\cap e_{1}^{\perp}=\{(0,x_{2},t)\in\R^{2}\times\R^{m};\|t\|_{\infty}\leq f(0,x_{2})\}.
\]
Thus,
\[
|K_{\infty,m}(f)\cap e_{1}^{\perp}|=2^{m}\int_{\R}f(0,x_{2})^{m}dx_{2}.
\]
Since, the dual of a section is the projection of the dual, one has
\[
\left(K_{\infty,m}(f)\cap e_{1}^{\perp}\right)^{\circ}=\left(K_{\infty,m}\left(f_{|e_{1}^{\perp}}\right)\right)^{\circ}=K_{1,m}\left(\L\left(f_{|e_{1}^{\perp}}\right)\right).
\]
Thus
\[
\left|\left(K_{\infty,m}(f)\cap e_{1}^{\perp}\right)^{\circ}\right|=\frac{2^{m}}{m!}\int_{\R}\left(\L\left(f_{|e_{1}^{\perp}}\right)(y)\right)^{m}dy.
\]
Our Theorem \ref{dim1} in dimension $1$ implies that
\[
\int_{\R}(f(0,x_{2}))^{m}dx_{2}\int_{\R}(\L(f_{|e_{1}^{\perp}})(y))^{m}dy\geq\frac{4}{m+1}.
\]
Hence,
\[
P(K_{\infty,m}(f)\cap e_{1}^{\perp})=\frac{4^{m}}{m!}\int_{\R}(f(0,x_{2}))^{m}dx_{2}\int_{\R}(\L(f_{|e_{1}^{\perp}})(y))^{m}dy\geq\frac{4^{m+1}}{(m+1)!}.
\]
 which implies that $K_{\infty,m}(f)\cap e_{1}^{\perp}$ verifies Mahler's conjecture in $\R^{m+1}$.\\
 Now, for $i\ge3$, again, by symmetry, we may assume that $i=3$. One has
 \[
 K_{\infty,m}(f)\cap e_{3}^{\perp}=\{(x,t)\in\R^{2}\times(\R^{m}\cap e_{3}^{\perp});\|t\|_{\infty}\leq f(x)\}=K_{\infty,m-1}(f).
 \]
 Hence,
\begin{eqnarray*}
P(K_{\infty,m}(f)\cap e_{3}^{\perp})&=&P(K_{\infty,m-1}(f))=\frac{4^{m-1}}{(m-1)!}\int_{\R^{2}}f(x)^{m-1}dx\int_{\R^{2}}(\L f(y))^{m-1}dy\\
&\geq&\frac{4^{m-1}}{(m-1)!}\frac{16}{m(m+1)}=\frac{4^{m+1}}{(m+1)!}.
\end{eqnarray*}
where the last inequality occurs by the induction hypothesis. 

\item Finally, let us show that $K_{\infty,m}(f)\cap e_{i}^{\perp}$, for any $i\in\{1,...,m+1\}$, can be partitioned into $2^{m+1}$ pieces of the same $(m+1)$-volume. For $i=1$ (and the same follows for $i=2$), since $f$ is even, one has
\begin{eqnarray*} 
\left|K_{\infty,m}(f)\cap e_{1}^{\perp}\cap\R^{m+1}_{+}\right|_{m+1}
&=&\left|\{(0,x_{2},t)\in\R^{2}_{+}\times\R^{m}_{+};\|t\|_{\infty}\leq f(0,x_{2})\}\right|_{m+1}\\
&=&\int_{0}^{+\infty}\left|f(0,x_{2})B_{\infty}^{m}\cap\R^{m}_{+}\right|_{m}dx_{2}\\
&=&\int_{0}^{\infty}f(0,x_{2})^{m}dx_{2}\\
%&=&\frac{1}{2}\int_{\R}f(0,x_{2})^{m}dx_{2}\\
&=&\frac{1}{2^{m+1}}\left|K_{\infty,m}(f)\cap e_{1}^{\perp}\right|_{m+1}.
\end{eqnarray*}
For $i=3$ (and the same follows for $i>3$), 
\[K_{\infty,m}(f)\cap e_{3}^{\perp}\cap \R^{m+1}_{+}=\{(x,t)\in\R^{2}_{+}\times(\R^{m}_{+}\cap e_{1}^{\perp});\|t\|_{\infty}\leq f(x)\}=K_{\infty,m-1}(f)\cap\R^{m+1}_{+}.
\]
Thus,
\[
\left|K_{\infty,m}(f)\cap e_{3}^{\perp}\cap\R^{m+1}_{+}\right|_{m+1}=\int_{\R^{2}_{+}}f(x)^{m-1}dx\times|B^{m-1}_{\infty}\cap\R^{m-1}_{+}|_{m-1}=\int_{\R^{2}_{+}}f(x)^{m-1}dx,
\]
and the result follows from the $m$-equipartition condition. 
\end{itemize}
\end{proof}

\begin{remark}
Notice that if one replaces in the above argument the cube $B_\infty^m$ by any unconditional convex body $L$ in $\R^m$, the same proof shows that $K_{L,m}(f)$ satisfies Mahler's conjecture in dimension $m+2$.
    
\end{remark}

\forget
\begin{conjecture}
Let $K$ be a symmetric convex body in $\R^{n}$. Then,
\[
\int_{K}m|t|^{m-1}dt\int_{K^{\circ}}m|t|^{m-1}dt\geq\int_{B_{\infty}^{n}}m|t|^{m-1}dt\int_{B_{1}^{n}}m|t|^{m-1}dt.
\]
\end{conjecture}
Since $K$ is a symmetric convex body, we know that there exists $a$, $b>0$ and two even concave functions $f_{1}$, $f_{2}:[-a,b]\to\R$ such that 
\[
K=\{(x,t)\in\R\times\R;\quad-a\leq x\leq b,\quad-f_{2}(x)\leq t\leq f_{1}(x)\}.
\]
Since $K$ is symmetric then we can consider that $-a\leq x\leq a$ and $-f_{2}(x)=-f_{1}(-x)$. Hence
\[
K=\{(x,t)\in[-a,a]\times\R;\quad -f(-x)\leq t\leq f(x)\}.
\]

\[
I=\int_{K}m|t|^{m-1}dt=\int_{-a}^{a}\int_{-f(-x)}^{f(x)}m|t|^{m-1}dtdx=2\int_{-a}^{a}\int_{f(a)\frac{x}{a}}^{f(x)}m|t|^{m-1}dtdx
\]
Let $b\in[-a,a]$ such that $f(b)=0$. Thus,
\begin{eqnarray*}
 I&=&2\int_{-a}^{0}\left[t^{m}\right]^{f(x)}_{f(a)\frac{x}{a}}dx+2\int_{0}^{b}\left[t^{m}\right]^{f(x)}_{0}dx+2\int_{0}^{b}\left[-(-t)^{m}\right]^{0}_{f(a)\frac{x}{a}}dx+2\int_{b}^{a}\left[-(-t)^{m}\right]^{f(x)}_{f(a)\frac{x}{a}}dx\\
 &=&2\int_{-a}^{0}\left(f(x)^{m}-\left(\frac{f(a)}{a}x\right)^{m}\right)dx+2\int_{0}^{b}f(x)^{m}dx+2\int_{0}^{b}\left(-f(a)\frac{x}{a}\right)^{m}dx\\
 &+&2\int_{b}^{a}\left[-(-f(x))^{m}+\left(-f(a)\frac{x}{a}\right)^{m}\right]dx\\
 &=&2\int_{-a}^{a}\mbox{sign}(f(x))|f(x)|^{m}dx-\left(\frac{f(a)}{a}\right)^{m}\left[\frac{x^{m+1}}{m+1}\right]^{0}_{-a}+2\left(\frac{-f(a)}{a}\right)^{m}\left[\frac{x^{m+1}}{m+1}\right]^{a}_{0}\\
 &=&2\int_{-a}^{a}\mbox{sign}(f(x))|f(x)|^{m}-\frac{|f(a)|^{m}a}{m+1}+\frac{|f(a)|^{m}a}{m+1}\\
 &=&2\int_{-a}^{a}\mbox{sign}(f(x))|f(x)|^{m}dx.
\end{eqnarray*}

\begin{theorem}
Let $K$ be a symmetric convex body in $\R^{2}$. Then,
\[
\int_{K}m|t|^{m-1}dt\int_{K^{\circ}}m|t|^{m-1}dt\geq\frac{16}{m+1}.
\]
\end{theorem}
\begin{proof}
Let $\mu(K)=\int_{K}m|t|^{m-1}dt$.\\
This is true due to Saint-Raymond.
Il faut introduire les shadow systems\\
t reste fixe et je deplace le x tout au long de la droite horizontale alors la mesure reste la meme ($\mu(K_{t})$ est cte) et $t\to\mu(K_{t}^{\circ})^{-1}$ est paire et convexe.
\end{proof}

\begin{theorem}
Let $K$ be a symmetric convex body in $\R^{3}$ and symmetric with respect to $e_{3}^{\perp}$. Set $\mu(K)=\int_{K}m|t|^{m-1}dtdx$, then
\[
\mu(K)\mu(K^{\circ})\geq\mu(B_{\infty}^{3})\mu(B_{1}^{3}).
\]
\end{theorem}
\newpage
\forgotten

\section{Proof of the inequalities in dimension 1 and 2 for $-\frac{1}{n}<s<0$.}
Let $-\frac{1}{n}<s=-\frac{1}{m}<0$ and $g:\R^{n}\to\R_{+}$ be an even $s$-concave function. Then, the function $f=g^s:\R^{n}\to(0,+\infty]$ is even and convex. In this section, we give a sharp lower bound of the operator $Q_{m}$ applied to even convex functions $f$ in dimensions $1$ and $2$.
First, we start by proving the following theorem, which is the analogue version of Theorem~\ref{th:mahler-s-F} for convex functions rather than $s$-concave ones, as seen in equation \eqref{eq:conjQmf-F}.
\begin{theorem}\label{th:m-neg-F-d2}
Let $f:\R^{2}\to (0,+\infty)$ be an even convex function such that $f\in\F$. Then, for any integer $m>2$ , one has
\[
\int_{\R^{2}}\frac{1}{(f(x))^{m}}dx\int_{\R^{2}}\frac{1}{(\M f(y))^{m}}dy\geq\frac{16m}{(m-1)(m-2)^2}.
\]
Moreover, there is equality for $f(x)=\max(1,\|x\|_\infty)$.
%with equality if and only if there exists a Hanner polytope $K\subset\R^{n+1}$ such that for every $x\in\R^{n}$, $f(x)=\|(x,1)\|_{K}=\inf\{\lambda>0; (x,1)\in\lambda K\}$.
\end{theorem}
\begin{proof}
Let us recall equation (\ref{eq:Mf}) and apply it in dimension $n$, which gives that, for all $m>n$ and for every symmetric convex body $K\in\R^{m-n}$, one has 
\[
\int_{\R^{n}}\frac{1}{f(x)^{m}}dx\int_{\R^{n}}\frac{1}{(\M f(y))^{m}}dy=\left(\frac{m}{m-n}\right)^{2}\frac{1}{|K||K^{\circ}|}|C_{m-n}^{K}(f)||C_{m-n}^{K}(f)^{\circ}|,
\]
where $C^{K}_{m-n}(f)=\left\{(x,s)\in\R^{n}\times\R^{m-n}; \|s\|_{K}f\left(\frac{x}{\|s\|_{K}}\right)\leq 1\right\}.$
Applying it for $K=B_{\infty}^{m-n}$, we simplify the notation of the associated convex body to 
\[
C_{m-n}^{\infty}(f)=\left\{(x,s)\in\R^{n}\times\R^{m-n};\|s\|_{\infty}f\left(\frac{x}{\|s\|_{\infty}}\right)\leq 1\right\}.
\]
And its dual is denoted by $C_{m-n}^1(\M f)$.
In addition, if $C_{m-n}^{\infty}(f)$ verifies Mahler's inequality in dimension $m$, we get
\begin{eqnarray*}
   \int_{\R^{n}}\frac{1}{f(x)^{m}}dx\int_{\R^{n}}\frac{1}{(\M f(y))^{m}}dy&=& \left(\frac{m}{m-n}\right)^{2}\frac{1}{|B_{\infty}^{m-n}||B_{1}^{m-n}|}|C_{m-n}^{\infty}(f)||C_{m-n}^{\infty}(f)^{\circ}|\\
   &\geq&\left(\frac{m}{m-n}\right)^{2}\frac{(m-n)!}{4^{m-n}}\frac{4^{m}}{m!}\\
  % &=&\frac{m}{m-n}\times\frac{(m-n-1)!}{(m-1)!}4^{n}\\
   &=&\frac{m}{m-n}\times\frac{4^{n}}{(m-n)\cdots(m-1)}.
\end{eqnarray*}
Thus, we get the desired inequality \eqref{eq:conjQmf-F} in dimension $n$. Since the $m$-equipartition condition, which is the key of our proof, is only verified in dimension $2$, it remains to prove that $C_{m-2}^{\infty}(f)$ verifies the hypothesis of Proposition \ref{prop:Mahler} in dimension $m$ for any integer $m\ge3$. 

We prove it by induction on $m$. For $m=3$, $C_{m-2}^{\infty}(f)=C_1^{\infty}(f)$ is a symmetric convex body in $\R^3$ hence it satisfies Mahler's conjecture by \cite{IS}. 

Let $m\ge4$ and assume that Mahler's conjecture is satisfied for $C_{m-3}^{\infty}(f)$ and let us prove it for $C_{m-2}^{\infty}(f)$.
For that, recall that $\R^{2}_{\varepsilon}=\{x\in\R^{2};\varepsilon_{i}x_{i}\geq 0,\forall i\in\{1,2\}\}$, $\forall\varepsilon\in\{-1,1\}^{2}$.
We can assume that the even convex function $f:\R^{2}\to\R_{+}$ is $m$-equipartitioned. In the sense that, $\forall\varepsilon\in\{-1,1\}^{2}$, one has
\[
\int_{\R_{\varepsilon}^{2}}\frac{dx}{f(x)^{m}}=\frac{1}{4}\int_{\R^{2}}\frac{dx}{f(x)^{m}}\quad\mbox{and}\quad\int_{\R_{\varepsilon}^{2}}\frac{dx}{f(x)^{m-1}}=\frac{1}{4}\int_{\R^{2}}\frac{dx}{f(x)^{m-1}}.
\]
We use proposition \ref{prop:Mahler} by proving that the convex body $C_{m-2}^{\infty}(f)$ verifies its conditions to deduce the lower bound of Mahler's volume of $C_{m-2}^{\infty}(f)$.
\begin{itemize}
\item \underline{$C_{m-2}^{\infty}(f)$ can be partitioned into $2^{m}$ pieces of the same volume.}
\[
C_{m-2}^{\infty}(f)\cap\R_{+}^{m}=\left\{(x,s)\in\R_{+}^{2}\times\R_{+}^{m-2};\|s\|_{\infty}f\left(\frac{x}{\|s\|_{\infty}}\right)\leq 1\right\}.
\]
Then,
\begin{eqnarray*}
\left|C_{m-2}^{\infty}(f)\cap\R_{+}^{m}\right|=\frac{m-2}{m}\left|B_{\infty}^{m-2}\cap\R_{+}^{2}\right|\int_{\R^{2}_{+}}\frac{dx}{f(x)^{m}}=\frac{m-2}{4m}\int_{\R^{2}}\frac{dx}{f(x)^{m}}.
\end{eqnarray*}
Thus,
\[
|C_{m-2}^{\infty}(f)|=\frac{m-2}{m}|B_{\infty}^{m-2}|\int_{\R^{2}}\frac{dx}{f(x)^{m}}=\frac{m-2}{m}2^{m-2}\int_{\R^{2}}\frac{dx}{f(x)^{m}}=2^{m}|C_{m-2}^{\infty}(f)\cap\R_{+}^{m}|.
\]
Similarly, since $f$ is even and $m$-equipartitioned, we can easily prove that for all $\varepsilon\in\{-1,1\}^{m}$ 
\[
|C_{m-2}^{\infty}(f)\cap\R_{\varepsilon}^{m}|=\frac{1}{2^{m}}|C_{m-2}^{\infty}(f)|.
\]

\item \underline{$C_{m-2}^{\infty}(f)\cap e_{i}^{\perp}$ verifies Mahler in dimension $m-1$, for $1\le i\le m$.}\\

For $i=1$ (the same holds for $i=2$),
\[C_{m-2}^{\infty}(f)\cap e_{1}^{\perp}=\left\{(0,x_{2},s)\in\R^{2}\times\R^{m-2};\|s\|_{\infty}f\left(\frac{(0,x_{2})}{\|s\|_{\infty}}\right)\leq 1\right\}=\{0\}\times C_{m-2}^{\infty}\left(f_{|e_{1}^{\perp}}\right).
\]
One has,
\[
\left(C_{m-2}^{\infty}(f)\cap e_{1}^{\perp}\right)^{\circ}=\left(\{0\}\times C_{m-2}^{\infty}\left(f_{|e_{1}^{\perp}}\right)\right)^{\circ}=\{0\}\times \left(C_{m-2}^{\infty}\left(f_{|e_{1}^{\perp}}\right)\right)^{\circ}=\{0\}\times C_{m-2}^{1}\left(\M \left(f_{|e_{1}^{\perp}}\right)\right).
\]
Thus, we get the following volumes
\[
\left|C_{m-2}^{\infty}(f)\cap e_{1}^{\perp}\right|=\left|C_{m-2}^{\infty}\left(f_{e_{1}^{\perp}}\right)\right|=\frac{m-2}{m-1}2^{m-2}\int_{\R}\frac{dx_{2}}{f_{|e_{1}^{\perp}}(x_{2})^{m-1}}.
\]
And
\[
\left|\left(C_{m-2}^{\infty}(f)\cap e_{1}^{\perp}\right)^{\circ}\right|=\left|C_{m-2}^{1}\left(\M\left(f_{|e_{1}^{\perp}}\right)\right)\right|=\frac{(m-2)2^{m-2}}{(m-1)(m-2)!}\int_{\R}\frac{dx_{2}}{\left(\M f_{|e_{1}^{\perp}}(x_{2})\right)^{m-1}}.
\]
Applying the result in dimension $1$, proved in \cite{FGSZ}, to $f_{|e_{1}^{\perp}}$ we get
\begin{eqnarray*}
    \left|C_{m-2}^{\infty}(f)\cap e_{1}^{\perp}\right|\left|\left(C_{m-2}^{\infty}(f)\cap e_{1}^{\perp}\right)^{\circ}\right|
    &=&\frac{(m-2)^{2}4^{m-2}}{(m-1)^{2}(m-2)!}\int_{\R}\frac{dx_{2}}{ f_{|e_{1}^{\perp}}(x_{2})^{m-1}}\int_{\R}\frac{dx_{2}}{\left(\M f_{|e_{1}^{\perp}}(x_{2})\right)^{m-1}}\\
    &\geq&\frac{(m-2)^{2}4^{m-2}}{(m-1)^{2}(m-2)!}\frac{m-1}{m-2}\frac{4}{m-2)}    =\frac{4^{m-1}}{(m-1)!}.
\end{eqnarray*}
For $i=3$ and the case $i>3$ follows similarly: 
\begin{eqnarray*}
C_{m-2}^{\infty}(f)\cap e_{3}^{\perp}=\left\{(x,0,s)\in\R^{2}\times\{0\}\times\R^{m-3};\|(0,s)\|_{\infty}f\left(\frac{x}{\|(0,s)\|_{\infty}}\right)\leq 1\right\}=C_{m-3}^{\infty}(f).
\end{eqnarray*}
Thus,
\[
\left|C_{m-2}^{\infty}(f)\cap e_{3}^{\perp}\right|=\left|C_{m-3}^{\infty}(f)\right|=\frac{m-3}{m-1}2^{m-3}\int_{\R^{2}}\frac{dx}{f(x)^{m-1}}.
\]
And
\[
\left|\left(C_{m-2}^{\infty}(f)\cap e_{3}^{\perp}\right)^{\circ}\right|=\left|\left(C_{m-3}^{\infty}(f)\right)^{\circ}\right|=\left|C_{m-3}^{1}(\M f)\right|=\frac{m-3}{m-1}\frac{2^{m-3}}{(m-3)!}\int_{\R^{2}}\frac{dy}{\left(\M f(y)\right)^{m-1}}.
\]
Then, using the hypothesis of induction on $m$, we obtain,
\begin{eqnarray*}
\left|C_{m-2}^{\infty}(f)\cap e_{3}^{\perp}\right|\left|\left(C_{m-2}^{\infty}(f)\cap e_{3}^{\perp}\right)^{\circ}\right|&=&\frac{(m-3)^{2}}{(m-1)^{2}}\frac{4^{m-3}}{(m-3)!}\int_{\R^{2}}\frac{dx}{f(x)^{m-1}}\int_{\R^{2}}\frac{dy}{\left(\M f(y)\right)^{m-1}}\\
&\ge&\frac{4^{m-1}}{(m-1)!}.
\end{eqnarray*}
which concludes the second part.
\item \underline{$C_{m-2}^{\infty}(f)\cap e_{i}^{\perp}$ can be partitioned into $2^{m-1}$ pieces of the same volume, for $1\le i\le m$.}\\

For $i=1$ and $i=2$, one has 
\[
C_{m-2}^{\infty}(f)\cap e_{1}^{\perp}=\{0\}\times C_{m-2}^{\infty}\left(f_{|e_{1}^{\perp}}\right)\quad\hbox{and}\quad 
C_{m-2}^{\infty}(f)\cap e_{2}^{\perp}= C_{m-2}^{\infty}\left(f_{|e_{2}^{\perp}}\right)\times\{0\}.
\]
These convex bodies are unconditional, thus we get the partition.

For $i=3$, by the $m$-equipartition condition of $f$, we obtain
\begin{eqnarray*}
\left|C_{m-2}^{\infty}(f)\cap e_{3}^{\perp}\cap\R_{+}^{m-1}\right|&=&\left|C_{m-3}^{\infty}(f)\cap\R_{+}^{m-1}\right|=\frac{m-3}{m-1}\left|B_{\infty}^{m-3}\cap\R_{+}^{m-1}\right|\int_{\R_{+}^{2}}\frac{dx}{f(x)^{m-1}}\\
&=&\frac{1}{4}\frac{m-3}{m-1}\int_{\R^{2}}\frac{dx}{f(x)^{m-1}}=\frac{1}{2^{m-1}}\left|C_{m-2}^{\infty}(f)\cap e_{3}^{\perp}\right|.
\end{eqnarray*}
 Using the same method, one may prove the partition of the body $C_{m-2}^{\infty}(f)\cap e_{i}^{\perp}$ for $i>3$.
\end{itemize}
Therefore $C_{m-2}^\infty(f)$ satisfies Mahler's conjecture, which concludes the proof of the inequality of Theorem~\ref{th:m-neg-F-d2}. Moreover, it is easy to check that there is equality for $f(x)=\max(1,\|x\|_\infty)$, since, in this case $C^\infty_{m-2}(f)=B_\infty^m$.
\end{proof}

Our next goal in this section is to prove the following theorem in dimension $1$, which establishes Mahler's conjecture in dimension 1 for $s$-concave functions for any $-1<s<0$.
\begin{theorem}\label{th:mahler-sneg}
Let $f:\R\to (0,+\infty]$ be an even convex function such that $\lim_{t\to+\infty}f(tx)=+\infty$, for any $x\neq 0$. Then, for any $m>1$, one has
\[
Q_m(f)=\int_{\R}\frac{dx}{f(x)^m}\int_{\R}\frac{dy}{\left(\M f(y)\right)^{m}}\geq Q_m\left(1+I_{[-1,1]}\right)=\frac{4}{m-1},
\]
with equality if and only if $f=c+I_{[-\alpha,\alpha]}$, for some $c,\alpha>0$.
\end{theorem}

We first show that Theorem \ref{th:mahler-sneg} follows from the next theorem, which implies Theorem \ref{comp-bipolar-one}, if one applies it to $f=g^s$, for any $g$, $s$-concave, for $-1<s<0$ and $m=-1/s$.
\begin{theorem}\label{th:f-mmf-dim1}
Let $f:\R\to (0,+\infty]$ be an even convex function such that $\lim_{t\to+\infty}f(tx)=+\infty$, for any $x\neq 0$. Then for any $m>1$, 
\begin{equation}\label{eq:f-mmf-dim1}
\int_{\R}\frac{1}{f(x)^m}dx\ge \frac{m-1}{m}\int_{\R}\frac{1}{(\M\M f(x))^m}dx,
\end{equation}
with equality if and only if $f=c+I_{[-\alpha,\alpha]}$, for some $c,\alpha>0$.
\end{theorem}
Multiplying equation \eqref{eq:f-mmf-dim1} by $\int(\M f)^{-m}$ we deduce that for any even convex $f$ tending to $+\infty$ at infinity, one has 
\begin{equation}\label{eq:Q-m-f-mmf}
Q_m(f)\ge \frac{m-1}{m} Q_m(\M f),
\end{equation}
with equality if and only if $f=c+I_{[-\alpha,\alpha]}$, for some $c,\alpha>0$.
Since $\M f\in\F$, using \cite{FGSZ} we get
\begin{equation}\label{fgsz-dim1}
Q_m(\M f)\ge Q_m\left(1+|x|\right)=\frac{4m}{(m-1)^2}.
\end{equation}
Using equations \eqref{eq:Q-m-f-mmf} and \eqref{fgsz-dim1}, we conclude the proof of Theorem~\ref{th:mahler-sneg}.
\begin{proof}[\bf{Proof of Theorem \ref{th:f-mmf-dim1}}]
Let us define a variant of the convex body $C_{m}^{K}$ introduced in section 2, which is more suitable for $m=1/|s|\notin\N$, this variant was considered in \cite{IW} for $s>0$:
\[
C_1(f)=C^{[-1,1]}_{1}(f)=\left\{(x,t)\in\R\times\R; \pi_1(f)(x,t)\leq 1\right\}=\left\{(x,t)\in\R\times\R; |t|f\left(\frac{x}{|t|}\right)\leq 1\right\}.
\]
The same computations as before show easily that $C_1(f)^\circ=C_1(\M f)$. Moreover, for any $m>1$, one has 
\[
\int_{\R}\frac{1}{f(x)^m}dx=\int_{\R}\int_{f(x)}^{+\infty}\frac{m}{t^{m+1}}dtdx=\int_0^{+\infty}\frac{m}{t^{m+1}}|\{x\in\R; f(x)\le t\}|dt.
\]
Using the change of variables $t=1/u$, we get 
\[
\int_{\R}\frac{1}{f(x)^m}dx=\frac{m}{2}\int_{C_1(f)}|u|^{m-2}dudx=\mu_m(C_1(f)),
\]
where we denote by $\mu_m$ the measure with density $d\mu_m(x,t)=\frac{m}{2}|t|^{m-2}dxdt$ on $\R^{2}$.
\forget
Consider the convex body in $\R^{2}$
\[
C_{1}(f)=\{(x,s)\in\R^{2}; |s|f\left(\frac{x}{|s|}\right)\leq 1\}.
\]
Since $f$ is even, then $C_{1}(f)$ is an unconditional convex body. In addition, one has 
\begin{eqnarray*}
    \int_{C_{1}(f)}|s|^{m-2}dsdx&=&2\int_{0}^{+\infty}s^{m-2}\left|\left\{x\in\R; \, (x,s)\in C_{1}(f)\right\}\right|ds \\
    &=&2\int_{0}^{+\infty}s^{m-2}\left|\{x\in\R; \, sf\left(\frac{x}{s}\right)\leq 1\}\right|ds\\
    &=& 2\int_{0}^{+\infty}s^{m-2}\left|\{sy\in\R; \, sf(y)\leq 1\}\right|ds\\
    &=& 2\int_{0}^{+\infty}s^{m-1}\left|\{y\in\R; \, sf(y)\leq 1\}\right|ds\\
    &=&2\int_{0}^{+\infty}\frac{1}{u^{m+1}}\left|\{y\in\R; \, f(y)\leq u\}\right|du.
\end{eqnarray*}
Now using Fubini, we obtain
\begin{eqnarray*}
    \int_{C_{1}(f)}|s|^{m-2}dsdx&=&2\int_{0}^{+\infty}\frac{1}{u^{m+1}}\int_{\{y\in\R; \, f(y)\leq u\}}dy du\\
    &=&2\int_{\R}\int_{f(y)}^{+\infty}\frac{1}{u^{m+1}}dudy\\
    &=&2\int_{\R}\left[\frac{u^{-m}}{-m}\right]_{f(y)}^{+\infty}dy\\
    &=&\frac{2}{m}\int_{\R}\frac{dy}{(f(y))^{m}}.
\end{eqnarray*}
Thus,

\[
\int_{\R}\frac{dy}{(f(y))^{m}}dy=\frac{m}{2} \int_{C_{1}(f)}|s|^{m-2}dsdx.
\]
\forgotten
Applied to $\M\M f$ instead of $f$, we get 
\[
\int_{\R}\frac{dz}{(\M \M f(z))^{m}}=\frac{m}{2}\int_{C_{1}(\M \M f)}|s|^{m-2}dsdx=\mu_{m}(C_{1}(\M\M f)).
\]
Since $C_{1}(\M \M f)=C_{1}(\M f)^{\circ}=C_{1}(f)^{\circ\circ}=\conv(C_{1}(f))$, we get
\[
\int_{\R}\frac{dz}{(\M \M f(z))^{m}}=\frac{m}{2}\int_{\conv(C_{1}(f))}|s|^{m-2}dsdx=\mu_m(\conv(C_{1}(f))).
\]
The inequality \eqref{eq:f-mmf-dim1} that we want to prove is thus equivalent to
\begin{equation}\label{eq:muC1fconv}
\mu_{m}(C_{1}(f))\geq\frac{m-1}{m}\mu_{m}\left(\conv(C_{1}(f))\right).
\end{equation}
It was proved in \cite[Theorem 2.16]{FGSZ} that $C_{1}(f)$ can be written as $C_{1}(f)=C_{1}(f)_{+}\cup C_{1}(f)_{-}$, where 
\[
C_{1}(f)_{+}=C_1(f)\cap\{s\ge0\} =\overline{\left\{(x,s)\in\R\times(0,+\infty); sf\left(\frac{x}{s}\right)\leq 1\right\}}
\]
is a convex body containing $0$ on its boundary and $C_{1}(f)_{-}$ is its symmetric image with respect to the hyperplane $\{s=0\}$. Thus, our problem can be formulated as follows:\\
for every convex body $K\subset\{(x,s)\in\R^{2}; s\geq 0\}$ such that $0\in\partial K$ and $K$ symmetric with respect to $\R e_{2}$ and if we denote $C=K\cup\sigma(K)$, where $\sigma=\sigma_{{e_{2}}^{\perp}}$ then 
\begin{equation}\label{eq:muC-muconvC}
\mu_{m}(C)\geq\frac{m-1}{m}\mu_{m}(\conv(C)).
\end{equation}
Let $a\in\R$ such that $P_{e_{2}^{\perp}}(K)=[-a,a]$, then there exists $b\geq 0$ such that $(a,b)\in K$. Set
\[
S=K\cap\{s=b\}=[-a,a]\times\{b\},\quad K_{1}=\conv(S,0),\quad  \mbox{and}\quad C_1=K_1\cup\sigma(K_1).
\]
Then, one has $C_1\subset C$ and one easily sees that $\conv(C)\setminus C\subset \conv(C_1)\setminus C_1$, thus 
\begin{equation}\label{eq:mu-C-C_1}
\frac{\mu_m(\conv(C)\setminus C)}{\mu_m(C)}\le \frac{\mu_m(\conv(C_1)\setminus C_1)}{\mu_m(C_1)}.
\end{equation}
\forget
\[
A=K\cap\{s\geq b\},\quad\quad\quad S=K\cap\{s=b\}\quad\mbox{ and }\quad K_{1}=\conv(A,0).
\]
It is easy to see that $K_{1}\subset K$ and $K_{1}=A\cup\conv(0,S)$. Since $\conv(K,\sigma(K))=A\cup\sigma(A)\cup\conv(S,\sigma(S))$, then $\conv(K_{1},\sigma(K_{1}))=\conv(K,\sigma(K))$.\\
Let $L$ be a symmetric convex body in $\R^{2}$ and $F$ be the function defined as follows:
\[
F(L)=\frac{\mu(L\cup\sigma(L))}{\mu(\conv(L,\sigma(L)))}.
\]
From the previous observation, one has $F(K)\geq F(K_{1})$.\\
Set $K_{2}=\conv(0,S)$. Then, $\conv(K_{2},\sigma(K_{2}))=\conv(S,\sigma(S))$. In addition, one has
\[
\mu(K_{1}\cup\sigma(K_{1}))=2\mu(K_{1})=2(\mu(A)+\mu(\conv(0,S)))=2(\mu(A)+\mu(K_{2})).
\]
On the other hand, using again the fact that $\conv(K_{1},\sigma(K_{1}))=A\cup\sigma(A)\cup\conv(S,\sigma(S))$, it implies that
\[
\mu_m(\conv(K_{1},\sigma(K_{1})))=2\mu_m(A)+\mu_m(\conv(K_{2},\sigma(K_{2}))).
\]
Thus,
\[
F(K_{1})=\frac{2\mu(A)+2\mu(K_{2})}{2\mu(A)+\mu(\conv(K_{2},\sigma(K_{2})))}\geq\frac{2\mu(K_{2})}{\mu(\conv(K_{2},\sigma(K_{2})))}=F(K_{2}).
\]
Our final step is to prove that $F(K_{2})\geq\frac{m-1}{m}$.
\forgotten
The triangle $K_1$ with vertices $0$, $(a,b)$ and $(-a,b)$ can be described as
\[
K_{1}=\{(x,s)\in\R^{2};\frac{b}{a}|x|\leq s\leq b\}.
\]
Hence, we get
\[
\mu_m(C_1)=2\mu_{m}(K_{1})=4\int_{0}^{a}\int_{\frac{b}{a}x}^{b}s^{m-2}dsdx=\frac{4}{m}ab^{m-1}.
\]
And since $\conv(C_1)=[-a,a]\times[-b,b]$, we get
\[
\mu_{m}(\conv(C_{1}))=4\int_{0}^{a}\int_{0}^{b}s^{m-2}dsdx=\frac{4}{m-1}ab^{m-1}.
\]
Hence 
\[
\frac{\mu_m(\conv(C_1)\setminus C_1)}{\mu_m(C_1)}=\frac{\frac{1}{m-1}-\frac{1}{m}}{\frac{1}{m}}=\frac{1}{m-1}.
\]
Using \eqref{eq:mu-C-C_1}, we deduce that \eqref{eq:muC-muconvC} holds. Moreover, there is equality in \eqref{eq:mu-C-C_1} if and only if $C=C_1$, i.e. $K=K_1$, which means that $f=\frac{1}{b}+I_{[-\frac{a}{b},\frac{a}{b}]}$. Thus, the equality case follows.
\forget
Then, $F(K_{2})=\frac{m-1}{m}$ and $F(K)\geq\frac{m-1}{m}$.\\
To deduce our inequality, we will use the observations done in the previous proof. In fact, we have
\begin{eqnarray*}
\int_{\R}\frac{1}{f(x)^m}dx\int_{\R}\frac{1}{(\M f(y))^{m}}dy&=&\left(\frac{m}{m-1}\right)^{2}\frac{1}{|B_{\infty}^{m-1}||B_{1}^{m-1}|}|(K\cup\sigma(K))||(K\cup\sigma(K))^{\circ}|\\
&=&\left(\frac{m}{m-1}\right)^{2}\frac{(m-1)!}{4^{m-1}}|(K\cup\sigma(K))||(K\cup\sigma(K))^{\circ}|\\
&=&\left(\frac{m}{m-1}\right)^{2}\frac{(m-1)!}{4^{m-1}}\frac{|K\cup\sigma(K)|}{|\conv(K,\sigma(K))|}|\conv(K,\sigma(K))||(K\cup\sigma(K))^{\circ}|
\end{eqnarray*}
Using Mahler's inequality in dimension $m$, one has 
\[
|\conv(K,\sigma(K))||(K\cup\sigma(K))^{\circ}|\geq\frac{4^{m}}{m!}.
\]
Thus, we get
\begin{eqnarray*}
 \int_{\R}\frac{1}{f(x)^m}dx\int_{\R}\frac{1}{(\M f(y))^{m}}dy&\geq&\left(\frac{m}{m-1}\right)^{2}\frac{(m-1)!}{4^{m-1}}\left(\frac{m-1}{m}\right)\frac{4^{m}}{m!} \\
 &=&\frac{4}{m-1},
\end{eqnarray*}
and the result follows.
\forgotten
\end{proof}
\forget
\section{Open problems}

For even $s$-concave function $g$ on $\R^n$.

Case $n\ge2$ and $1/|s|\notin\N$. 

Case $n\ge2$, $s<0$ and $f\notin\F$: proof of Theorem 4.3

equality cases. 

We simplify the notation and denote:
\[
C_1(f)=C^{[-1,1]}_{1}(f)=\{(x,s)\in\R\times\R; \pi_1(f)(x,s)\leq 1\}=\{(x,s)\in\R\times\R; |s|f\left(\frac{x}{|s|}\right)\leq 1\}.
\]
We know that $C_1(f)^\circ=C_1(\M f)$. Moreover, for any $m>n$, one has 
\[
\int_{\R^n}\frac{1}{f(x)^m}dx=m\int_{\R^n}\int_{f(x)}^{+\infty}\frac{m}{s^{m+1}}dsdx.
\]
Using the change of variables $s=1/t$ and $x=y/t$, we get 
\[
\int_{\R^n}\frac{1}{f(x)^m}dx=\frac{m}{2}\int_{C_1(f)}|t|^{m-n-1}dtdx=\frac{m}{2}\mu_m(C_1(f)),
\]
where we denote by $\mu_m$ the measure on $\R^n\times\R$ with density $|t|^{m-n-1}$, where $t$ denotes the last coordinate.\\
\forgotten

\noindent {\bf Acknowledgments:}  The authors  are grateful to the anonymous referee for a careful reading of the paper and constructive comments and corrections and they also thank Nathaël Gozlan for his remarks and pertinent questions on this work.

\noindent
{\footnotesize\sc Matthieu Fradelizi:}
  {\footnotesize Univ Gustave Eiffel, Univ Paris Est Creteil, CNRS, LAMA UMR8050 F-77447 Marne-la-Vallée, France. \\
  ORCID: 0000-0001-9362-6819}\\[-1.3mm]
  {\footnotesize e-mail: {\tt matthieu.fradelizi@univ-eiffel.fr \tt }}\\

\noindent
{\footnotesize\sc Elie Nakhle:}
  {\footnotesize Univ Paris Est Creteil, Univ Gustave Eiffel, CNRS, LAMA UMR8050, F-94010 Creteil, France. }\\[-1.3mm]
  {\footnotesize e-mail: {\tt elie\_b\_nakhle@hotmail.com \tt }}\\

\end{document}